\documentclass[10pt]{article}

\usepackage{graphicx}  % For images
\usepackage[dvipsnames]{xcolor}  % Enhanced color options
\usepackage{float}  % Better figure placement
\usepackage{booktabs}  % Better table formatting
\usepackage{longtable}  % Multi-page tables
\usepackage{makecell}  % Custom cell formatting in tables
\usepackage{lmodern}  % Improved font rendering
\usepackage{geometry}  % Control margins
\usepackage{amsmath, amsfonts, amssymb, amsthm, mathtools}

\newcommand{\E}{\mathbb{E}}  % Expectation operator

\usepackage{amsmath, amsfonts, amssymb, amsthm, mathtools}

\newtheorem{theorem}{Theorem}[section]
\newtheorem{lemma}[theorem]{Lemma}
\newtheorem{proposition}[theorem]{Proposition}
\newtheorem{corollary}[theorem]{Corollary}

\theoremstyle{definition}

\newtheorem{example}[theorem]{Example}

\theoremstyle{remark}
\newtheorem{remark}[theorem]{Remark}

\usepackage{caption}
\usepackage{subcaption}
\usepackage{pdfpages}  % Insert PDFs as images
\usepackage{pgfplots}
\pgfplotsset{compat=1.18}  % Ensure compatibility with latest pgfplots
\usepgfplotslibrary{dateplot}  % Additional plotting features

\usepackage{comment}  % Comment out blocks of text
\usepackage{xspace}  % Handle spacing in macros
\usepackage{titlesec}  % Section title formatting
\usepackage[section]{placeins}  % Ensure floats stay within sections
\usepackage{setspace}  % Adjust line spacing
\usepackage{ulem}  % Underlining and strikethrough
\usepackage{pdflscape}  % Rotating pages (alternative to rotating)
\usepackage{hyperref}  % Clickable links
\hypersetup{colorlinks=true, linkcolor=blue, citecolor=blue, urlcolor=black}

\usepackage[numbers,square]{natbib}
\setcitestyle{square}

\def\phi{{\varphi}}

\def\be{\begin{equation}}
\def\ee{\end{equation}}
\def\ber{\begin{eqnarray}}
\def\eer{\end{eqnarray}}

\begin{document}

\addtolength{\textheight}{0 cm} \addtolength{\hoffset}{0 cm}
\addtolength{\textwidth}{0 cm} \addtolength{\voffset}{0 cm}

\newenvironment{acknowledgement}{\noindent\textbf{Acknowledgement.}\em}{}

\setcounter{secnumdepth}{5}

 % --------- AMS & Math Packages ----------

% Define expectation symbol
% \newcommand{\E}{\mathbb{E}}  % Expectation operator

\newtheorem*{assumption}{Assumptions on $a(x)$}

\newtheorem*{thm*}{Theorem A}

\title{ Positive singular solutions of a certain elliptic PDE}
\author{N. Mohammadnejad\footnote{University of Alberta, Edmonton Alberta, Canada, negar3@ualberta.ca} \quad    }

%\author{A. Aghajani}
%\address{School of Mathematics, Iran University of Science and Technology, Narmak, Tehran, Iran }
%\email{aghajani@iust.ac.ir}
%\urladdr{www.math.sc.edu/$\sim$howard} % Delete if not wanted.

\date{}

\maketitle

\maketitle  % Generates title
% \tableofcontents
\begin{abstract}

In this paper, we investigate the existence of positive singular solutions for a system of partial differential equations on a bounded domain \begin{equation} \label{main equation of the thesis}
\left\{
\begin{array}{lr}
-\Delta u = (1+\kappa_1(x)) | \nabla v |^p & \text{in}~~ B_1 \backslash \{0\},\\
-\Delta v = (1+\kappa_2(x)) | \nabla u |^p & \text{in}~~ B_1 \backslash \{0\},\\
u = v = 0 & \text{on}~~ \partial B_1.
\end{array}
\right.
\end{equation}   
We investigate the existence of positive singular solutions within $B_1$, the unit ball centered at the origin in $ \mathbb{R}^N $, under the conditions $ N \geq 3 $ and  $\frac{N}{N-1} < p < 2 $. Additionally, we assume that $ \kappa_1$  and $\kappa_2$  are non-negative, continuous functions satisfying  $\kappa_1(0) = \kappa_2(0) = 0$ . Our system is an extension of the PDE equation studied by Aghajani et al. \cite{AghajaniA.2021Sepi} under similar assumptions. 
\end{abstract}

% -------- INTRODUCTION --------
\section{Introduction}

Numerous studies have been done on the existence of positive singular solutions for various partial differential equations. One such example is the Lane-Emden equation.
\begin{equation} \label{base equation}
\left\{
\begin{array}{lr}
-\Delta u= u^p& \text{in}~~~~  \Omega,\\
u=0 & \text{on}~~ \partial \Omega,
\end{array}
\right.
\end{equation}

For \( p > 1 \), let \( \Omega \) be a bounded domain in \( \mathbb{R}^N \) (\( N \geq 3 \)) with a smooth boundary. The existence and non-existence of solutions to this problem have been extensively studied on various bounded domains and for different ranges of \( p \); see \cite{MR0192184,MR1998495,MR2057495,MR615628,MR762722}. A closely related example to equation \eqref{main equation of the thesis} is equation \begin{equation} \label{base equation}
\left\{
\begin{array}{lr}
-\Delta w = | \nabla w |^p & \text{in}~~ B_1 \backslash \{0\},\\
w = 0 & \text{on}~~ \partial B_1.
\end{array}
\right.
\end{equation} 
The case \( 0 < p < 1 \) was analyzed in \cite{AghajaniA.2018Ssoe}, while boundary blow-up versions of \eqref{base equation}, where the negative sign in front of the Laplacian is removed, were studied in \cite{MR990591,MR2488065}. Additionally, various studies on equations similar to \eqref{base equation} can be found in \cite{MR2718666,MR2514734,MR3229851,MR3426084,MR2193390}.  \\
In this work, we aim to establish the existence of positive singular solutions to equation \eqref{main equation of the thesis} in the unit ball \( B_1 \), where \( \frac{N}{N-1} < p < 2 \) and \( B_1 \) is centered at the origin in \( \mathbb{R}^N \) (\( N \geq 3 \)). Unlike previous studies, we impose no smallness conditions on \( \kappa_1 \) and \( \kappa_2 \) beyond the assumption that they vanish at the origin.  

Our main result is as follows.  

\begin{theorem}\label{main_theorem}
Suppose  $N\geq 3$, $\frac{N}{N-1}<p<2$ and $\kappa_1, \kappa_2$ are non-negative, continuous functions with $\kappa_1(0)=\kappa_2=0$. There exists an infinite number of positive singular solutions $u_t,v_t$ of
\begin{equation} \label{eq_syst}
\left\{
\begin{array}{lr}
-\Delta u= (1+\kappa_1(x)) | \nabla v|^p& \text{in}~~ B_1 \backslash \{0\},\\
-\Delta v= (1+\kappa_2(x)) | \nabla u|^p& \text{in}~~ B_1 \backslash \{0\},\\
u=v=0 & \text{on}~~ \partial B_1.
\end{array}
\right.
\end{equation} which blow up at the origin. Moreover, $u_t$ and $ v_t$ converge uniformly to zero away from the origin as $t$ approaches infinity. 
\end{theorem}
Equation \eqref{eq_syst} is an extension of the equation studied by Aghajani et al. \cite{AghajaniA.2021Sepi} where they analyzed the existence of positive singular solutions on bounded domains and also classical solutions on exterior domains. 
%$N\geq 3$. Then for $\frac{N}{N-1}<p<2$ we define $\xi \coloneqq (p-1)(N-1), , \sigma \coloneqq \frac{2-p}{p-1}$ and note  $\xi >1$ }
%\[w_t(r) =\int_r^1 \frac{1}{(ty^{\xi}+\beta y)^{\frac{1}{p-1}}} \quad t>-\beta \]
\noindent
\textbf{The parameters.} The parameters $p, \xi,\beta,$ and $\sigma$ for the remainder of this work will be as follows unless otherwise stated.\\
\noindent
First we consider the following singular solution of a scalar problem. 
\begin{example} \label{example} Let $N \ge 3 $, $\frac{N}{N-1}<p<2$,  $\xi := (p-1)(N-1)$ (note this implies $ \xi>1$), $ \beta:=\frac{p-1}{\xi-1} >0$. Then for all $t \ge 0$ the radial function 
\begin{equation} \label{integral_solution}
w_{t}(r)=\int_{r}^{1} \dfrac{dy}{({ty^{\xi} + \beta y)}^{\frac{1}{p-1}}}
\end{equation} is the singular solution of 
\begin{equation} \label{eq_ball_scal}
\left\{
\begin{array}{lr}
-\Delta w= | \nabla w|^p& \text{in}~~ B_1 \backslash \{0\},\\
w=0 & \text{on}~~ \partial B_1.
\end{array}
\right.
\end{equation}
\end{example} 

\begin{remark}
    For the remaining sections of this work that focus on results in bounded domains, we adopt the parameter values from Example \ref{example}. This also applies to all the material presented in the Introduction.
\end{remark}
\noindent \textbf{Acknowledgement}
     This project was done during my master's thesis. I thank my advisor Dr. Craig Cowan for suggesting  this topic to me and for his support, understanding and valuable guidance

% -------- MATHEMATICAL SECTION --------

% -------- FIGURES --------
\subsection{Strategy} Our approach for finding a solution $ (u,v)$ of \eqref{eq_syst} will be to look for solutions of the form
$ u=w_t+\phi$ and $ v= w_t+ \psi$. Note that $ \phi,\psi$ are unknown functions, but they are equal to zero on the boundary of $ B_1$, and $w_t$ is the solution of \eqref{eq_ball_scal}. By considering our solutions $u$ and $v$ in this new form, $ \phi$ and $\psi$ need to satisfy 
\begin{equation} \label{eq_phi_psi}
\left\{
\begin{array}{lr}
-\Delta \phi - p | \nabla w_t|^{p-2} \nabla w_t \cdot \nabla \psi= \kappa_1(x) | \nabla w_t + \nabla \psi|^p + I(\psi)& \text{in}~~ B_1 \backslash \{0\},\\
-\Delta \psi - p | \nabla w_t|^{p-2} \nabla w_t \cdot \nabla \phi= \kappa_2(x) | \nabla w_t + \nabla \phi|^p + I(\phi)& \text{in}~~ B_1 \backslash \{0\},\\
\phi=\psi=0 & \text{on}~~ \partial B_1,
\end{array}
\right.
\end{equation}where 
\[ I(\zeta) \coloneqq | \nabla w_t + \nabla \zeta|^p - | \nabla w_t|^p - p | \nabla w_t|^{p-2} \nabla w_t \cdot \nabla \zeta.\] 

\noindent To find a solution of \eqref{eq_phi_psi}, we will apply a fixed point argument. A key step will be to understand the linear operator on the left hand side of \eqref{eq_phi_psi}, namely the solvability of 
\begin{equation}  \label{linear_syst}
\left\{
\begin{array}{lr} 
-\Delta \phi-p | \nabla w_t|^{p-2} \nabla w_t \cdot \nabla \psi=f& \text{in}~~ B_1 \backslash \{0\},\\
-\Delta  \psi - p | \nabla w_t|^{p-2} \nabla w_t \cdot \nabla \phi=g& \text{in}~~ B_1 \backslash \{0\},\\
\phi=\psi=0 & \text{on}~~ \partial B_1,
\end{array}
\right.
\end{equation} for $(\phi,\psi)$, given $ f $ and $g$.   

 \noindent We define the nonlinear mapping $T_t(\phi,\psi)= ( \hat{\phi}, \hat{\psi})$ via 
\begin{equation} \label{non-map1}
\left\{
\begin{array}{lr}
-\Delta \hat \phi - p | \nabla w_t|^{p-2} \nabla w_t \cdot \nabla \hat \psi= \kappa_1(x) | \nabla w_t + \nabla \psi|^p + I(\psi)& \text{in}~~ B_1 \backslash \{0\},\\
-\Delta \hat \psi - p | \nabla w_t|^{p-2} \nabla w_t \cdot  \nabla \hat \phi= \kappa_2(x) | \nabla w_t + \nabla \phi|^p + I(\phi)& \text{in}~~ B_1 \backslash \{0\},\\
\phi=\psi=0 & \text{on}~~ \partial B_1.
\end{array}
\right.
\end{equation} We will show that$T_t$ is a contraction on a suitable complete metric space. Subsequently, we will use this result to prove the existence of the positive singular solutions of \eqref{eq_syst}.  
\noindent 
We fix $ (f,g)$, and set $F \coloneqq f+g$ and $ G \coloneqq f-g$. Suppose that $ \zeta_i$ is a solution of the scalar problems
\begin{equation} \label{scal_1} \left\{
\begin{array}{lr}
-\Delta \zeta_1-p | \nabla w_t|^{p-2} \nabla w_t \cdot \nabla \zeta_1=F& \text{in}~~ B_1 \backslash \{0\},\\
\zeta_1=0 & \text{on}~~ \partial B_1,
\end{array}
\right.
\end{equation} 

\begin{equation} \label{scal_2} \left\{
\begin{array}{lr}
-\Delta \zeta_2+p | \nabla w_t|^{p-2} \nabla w_t \cdot \nabla \zeta_2=G& \text{in}~~ B_1 \backslash \{0\},\\
\zeta_2=0 & \text{on}~~ \partial B_1.
\end{array}
\right.
\end{equation} 
%where 
%\[F= \kappa_1(x)|\nabla w + \nabla \psi |^p+  | \nabla w + \nabla \zeta|^p - | \nabla w|^p - p | \nabla w|^{p-2} \nabla w \cdot \nabla \zeta\]
%\[G= \kappa_2(x)|\nabla w + \nabla \psi |^p+  | \nabla w + \nabla \zeta|^p - | \nabla w|^p - p | \nabla w|^{p-2} \nabla w \cdot \nabla \zeta\]
The problem in (\ref{scal_1}) has been studied in the previous work \cite{AghajaniA.2021Sepi}. In this work we will prove the results for(\ref{scal_2}). A linear algebra argument shows that if $ \zeta_i$ satisfies the above, then for $(\phi,\psi)$ to be a solution of (\ref{linear_syst}), we need $ \phi,\psi$ to satisfy $ \phi+ \psi=\zeta_1, $ and $ \phi - \psi= \zeta_2$. From this, we see that 
\begin{equation} \label{so_100}
\phi = \frac{\zeta_1 + \zeta_2}{2} \quad \text{and} \quad \psi = \frac{\zeta_1 - \zeta_2}{2}.
\end{equation}   So to understand the solvability of \eqref{linear_syst}, it is enough to understand the solvability of the two scalar problems given by \eqref{scal_1} and \eqref{scal_2}.  
One can write out the left hand sides of \eqref{scal_1} and \eqref{scal_2} explicitly respectively as 
\[ -\Delta \zeta_1 + \frac{p}{(t|x|^{\xi-1}+\beta )}  \frac{x \cdot \nabla \zeta_1}{|x|^2}, \quad  -\Delta \zeta_2 - \frac{p}{(t|x|^{\xi-1}+\beta )}  \frac{x \cdot \nabla \zeta_2}{|x|^2}.   \] 
This motivates the definition of the linear operators
\begin{equation}\label{linear_operator}
    L_t^\pm(\zeta)= -\Delta \zeta \pm \frac{p}{(t|x|^{\xi-1}+\beta )}  \frac{x \cdot \nabla \zeta}{|x|^2}.
\end{equation}
 So $L_t^+$ is the linear operator associated with the left hand side of (\ref{scal_1}) and $L_t^-$ is the linear operator associated with the left-hand side of (\ref{scal_2}).
Later we will need the following asymptotic result:  
\begin{align} \label{asymp_1}
  &\bullet \lim_{r \to 0}r^{\sigma+1}w'_t(r)=-C_\beta~~~~ \text{where }~~~C_\beta=\frac{1}{\beta^{\frac{1}{p-1}}}\\ \nonumber
&\bullet \mbox{ for all $ t \ge 0$, there exists a constant $C$ such that }  \lim_{r \searrow 0} r^\sigma w_t(r) = C. \end{align}

 \section{The linear theory}
%\begin{theorem} Suppose $L_t:X \rightarrow Y$ is o   
%\end{theorem}
\subsection{Analysis of the Linear operators $L_t^{\pm}$}
Define the following norms
\[ \| \phi \|_X \coloneqq  \underset{0<|x|\leq 1}{\sup}  \{ |x|^{\sigma} |\phi| +  |x|^{\sigma+1} |\nabla \phi|\}, \quad \|f \|_Y \coloneqq  \underset{0<|x|\leq 1}{\sup}  \{ |x|^{\sigma+2} |f(x)| \} \]
 where $ \sigma= \frac{2-p}{p-1}$. 
Let $X$ denote the set of functions  $\phi$ such that $\| \phi \|_X < \infty$  and vanish on  the boundary of $B_1$ and let $Y$ denote the set of functions $f$ such that $\|f\|_Y< \infty$.  
\noindent
The goal is to show that our nonlinear mapping $T_t(\phi, \psi)$ is a contraction by applying Banach's Fixed Point Theorem on the complete metric space 
\[ \mathcal{F}_R \coloneqq \left\{ (\phi,\psi) \in X \times X:  \| \phi\|_X, \|  \psi\|_X \le R     \right\}.   \]
On this space, we have
\[ \|( \phi,\psi)\|_{X \times X}:= \| \phi \|_X + \| \psi \|_X .\]
We can now state our main proposition in this chapter.
\begin{proposition} \label{the_real_main_result}
     Let $N \ge 3 $,  such that $\frac{N}{N-1}<p<2$,  $\xi = (p-1)(N-1)>1$, $ \beta=\frac{p-1}{\xi-1} >0$ and  $\sigma =\frac{2-p}{p-1}$. There is a positive constant $C$ such that for all  $ f,g \in Y$, there are functions $ \phi $ and $\psi $ in $X$ which satisfy the equation \begin{equation}  \label{linear_beg2}
\left\{
\begin{array}{lr}
-\Delta \phi-p | \nabla w_t|^{p-2} \nabla w_t \cdot \nabla \psi=f& \text{in}~~ B_1 \backslash \{0\},\\
-\Delta  \psi - p | \nabla w_t|^{p-2} \nabla w_t \cdot \nabla \phi=g& \text{in}~~ B_1 \backslash \{0\},\\
\phi=\psi=0 & \text{on}~~ \partial B_1,
\end{array}
\right.
\end{equation} and the estimate
\[ \| \phi \|_X + \| \psi \|_X \le C \|f\|_Y + C \|g\|_Y.\] 
\end{proposition}
\noindent
With the approach we have taken in this section, we use the change of notation from \eqref{so_100} and restate our proposition as follows.

\begin{proposition}\label{main_result}  Let $N \ge 3 $,  such that $\frac{N}{N-1}<p<2$,  $\xi = (p-1)(N-1)>1$, $ \beta=\frac{p-1}{\xi-1} >0$ and  $\sigma =\frac{2-p}{p-1}$.  There is some positive constant $C$ such that for all  $ f \in Y$ and non-negative $ t $, there is some $ \phi \in X$ which satisfies 
\begin{equation}  \label{eq_equation}
\left\{
\begin{array}{lr}
L_t^\pm(\phi)=f & \text{in}~~ B_1 \backslash \{0\},\\
\phi =0& \text{on}~~ \partial B_1.
\end{array}
\right.
\end{equation} Moreover, one has the estimate $\| \phi \|_X \leq C\| f\|_Y$.
\end{proposition}
\noindent
\textbf{Spherical harmonics}.  Note $\Delta_\theta= \Delta _{S^{N-1}}$  is the Laplace-Beltrami operator on $S^{N-1}$ with eigenpairs $(\psi_k, \lambda_k)$ satisfying
\[-\Delta_\theta \psi_k(\theta)= \lambda_k\psi_k(\theta)~~~~  \text{for} ~~~  \theta \in S^{N-1}\]    and $\psi_0=1$,  $\lambda_0=0, \lambda_1=N-1, \lambda_2=2N$. We normalize $\psi_k$ in $L^2(S^{N-1})$ such that $\| \psi_k \|_{L^2(S^{N-1})}=1$.\\ 
 Using this, given $f \in Y, $ and $ \phi \in X$, we can decompose $ \phi$ and $f$ into various modes by writing\[
 f(x) = \sum_{k=0}^\infty b_k(r) \psi_k(\theta) \ ~~~~ \text{and} ~~~~ \phi(x) = \sum_{k=0}^\infty a_k(r) \psi_k(\theta).\]
 Note that $a_k(1)=0$ after considering the boundary condition of $ \phi$. A computation shows that $\phi$ satisfies \eqref{eq_equation} provided $ a_k$ satisfies 
\begin{equation} \label{ode_k_mode_b}
-a''_k(r)-\frac{(N-1)a'_k(r)}{r} + \frac{\lambda_k a_k(r)}{r^2} \pm \frac{pa'_k(r)}{\beta r +t r^{\xi}}=b_k(r) \quad \text{for}~~ 0<r<1 
\end{equation} with $ a_k(1)=0$.\\
Let $X_1$ and $Y_1$  be the closed subspaces of $X$ and $Y$ defined by 
\[X_1 \coloneqq \left\{ \phi \in X : \phi =\sum^{\infty}_{k=1} a_k(r)\psi_k(\theta) \right\}, \quad  \quad Y_1 \coloneqq \left\{ f \in Y : f=\sum^{\infty}_{k=1} b_k(r)\psi_k(\theta) \right\}. \] Note that $X_1$ and $Y_1$ are just the same representations as $X$ and $Y$ except that the summations start at $k=1$ and not $k=0$. \\
To show that they are closed subspaces of $X$ and $Y$, we need to show that if any sequence $\phi_m \in X_1$ is such that $\phi_m $ converges to some $\phi$ in $X$, then $\phi $ is in $X_1$. We need a similar result for $Y$ and $Y_1$. We first note that if we have $\phi(x) \in X_1$ then we can write $\phi(r, \theta)=\sum^{\infty}_{k=1} a_k(r)\psi_k(\theta)$ 
and by integrating both sides  with respect to $\theta$, we get \[\int_{S^{N-1}} \phi(r, \theta) d \theta=\sum^{\infty}_{k=1} a_k(r) \int_{S^{N-1}}\psi_k(\theta) d \theta=0.\]
This is due to the orthonormality of ${\psi_k}'s$ and $\psi_0=1$. 
Thus, we can conclude that for all $\phi $ in $X$ we have
\begin{equation}
    \label{closed subspace}
    \phi \in X_1 \Longleftrightarrow \int_{S^{N-1}} \phi(r, \theta)=0 ~~~ \text{for all}~~ 0<r \le1.
\end{equation}
Since by the assumption we have $\phi_m(r,\theta) \in X_1 $, it follows that $\int_{S^{N-1}} \phi_m(r, \theta)=0$ for all $0<r\le 1$. Fix $0<r\le 1$ and
 set $ \zeta_m(\theta) = \phi_m(r, \theta)$ and $ \zeta(\theta) = \phi(r ,\theta)$.  We claim that $ \zeta_m $  converges uniformly to $\zeta$ in $X$ on $S^{N-1}$.
Assuming the claim is true, we can fix $0<r\le 1$
and by the uniform convergence of $\zeta_m$ to $
\zeta$ on $S^{N-1}$ we get 
\[0=\int_{S^{N-1}} \phi_m(r ,\theta )\longrightarrow \int_{S^{N-1}} \phi(r, \theta).\]
Thus we have $\int_{S^{N-1}} \phi(r, \theta)=0$ and by \eqref{closed subspace} we can see that $\phi \in X_1$. 
We now prove the claim. If we fix  $\epsilon>0$, we can write \begin{align*}
 \| \phi_m - \phi \|_X  \geq \underset{0<|x|\leq 1}{\sup}|x|^\sigma
    |\phi_m( x, \theta) - \phi(x, \theta)|\geq \underset{\epsilon<|x|\leq 1}{\sup}|x|^\sigma
    |\phi_m( x, \theta) - \phi(x,\theta)|\geq {\epsilon
    }^\sigma 
    |\phi_m( x,\theta) - \phi(x,\theta)|.
\end{align*}
Note that $r=|x|$  and $\epsilon$ are fixed and we know that since $\phi_m$ converges to $\phi$ in $X$  we have $\| \phi_m - \phi \|_X \to 0$. So we get the uniform convergence of $\phi_m$ to $\phi$ away from the origin. We can use a similar proof for $Y_1$. Thus $X_1, Y_1$ are closed subspaces of $X$ and $Y$ respectively.

%$\beta_k^+(-)$ 

\noindent
\subsection{Kernel of $L^\pm_0$}
 Using the definition of our linear operator \eqref{linear_operator} in the case of $t=0$, we can write it explicitly as
\[ L_0^\kappa(\phi)= -\Delta \phi + \frac{\kappa x \cdot \nabla \phi(x)}{ \beta |x|^2}\]  
where $\kappa$ could be either equal to $+p$ or $-p $. For simplicity, we set
$ L_0^{+p}\coloneqq L_0^+ $ and $L_0^{-p} \coloneqq L_0^-  $.
\begin{lemma}  \label{L_0 kernel lemma}Let $\frac{N}{N-1}<p<2$,  $\xi = (p-1)(N-1)>1$, $ \beta=\frac{p-1}{\xi-1} >0$,  $\sigma =\frac{2-p}{p-1}$ and suppose $ \phi\in X_1$ is such that $L_0^\kappa(\phi)=0$ in $B_R \backslash \{0\}$ with $\phi=0$ on $\partial B_R$ in the case of $R$ finite where we have $\kappa=p$ or $\kappa=-p$, then $\phi=0$.
\end{lemma}
\begin{proof} We saw that we can decompose $\phi(x)$ into $\phi(x)= \sum_{k=1}^{\infty} a_k(r) \psi_k(\theta),$ and since $L_0^\kappa=0$ in $B_R \backslash \{0\}$, then
 $a_k$ should satisfy
\begin{equation}\label{kernel L_0 equation}
     -a''_k(r)-\frac{(N-1)a'_k}{r} + \frac{\lambda_k a_k(r)}{r^2} +\frac{\kappa a'_k(r)}{\beta r}  =0 \quad \text{for}~~ 0<r<R
\end{equation}
where $a_k(R)=0$ in the case of $R< \infty$. Also we have \begin{equation}\label{a_k in L_0 statement}
    \underset{0<r<R }{\sup}  \{ r^{\sigma} |a_k(r)| +  r^{\sigma+1} |a'_k (r)|\}< \infty.
\end{equation}
\textit{Proof of the statement.} To show that \eqref{a_k in L_0 statement} is true, note that 
given $ \phi \in X$, we have 
$ \phi(x) = \sum_{k=0}^\infty a_k(r) \psi_k(\theta)$ with $a_k(R)=0$. By using the property of orthonormality of  ${\psi_k(\theta)}'s$,  we get 
\[ a_k(r) = \int_{ S^{N-1}} \phi(r, \theta) \psi_k(\theta) d \theta. \] 
Thus, noting that $|\phi(r,\theta)|r^\sigma \leq \|\phi\|_X$, we obtain
\[|a_k(r)| \leq \underset{S^{N-1}}{\sup}|\psi_k(\theta)| \int_{|\theta|=1} \frac{\|\phi\|_X}{r^\sigma}d \theta.\] Since $\phi$ is in $X$, there exists a positive $C_k$ such that $\|\phi\|_X \leq C$. Thus we get $|a_k(r)| \leq \frac{C_k}{r^\sigma}$ and we deduce \begin{equation}\label{firsta_k inequality}
    r^\sigma |a_k(r)| \leq C_k.
\end{equation}
Using the definition of the gradient in n-dimensional spherical coordinates and the orthogonal properties, we obtain 
$\phi_r= \nabla \phi(x) \cdot \hat{r}$ which gives us $ |\phi_r| \leq |\nabla \phi|  .$
Again for $\phi \in X  $  we can write $\phi_r(r ,\theta) = \sum_{k=0}^\infty a'_k(r) \psi_k(\theta)$ where $a_k(1)=0$. We use the orthonormality of ${\psi_k(\theta)}'s$, and we obtain
\[ a'_k(r) = \int_{S^{N-1}} \phi_r(r ,\theta) \psi_k(\theta) d \theta. \] 
Thus, noting that we have $|\nabla\phi(r,\theta)|r^{\sigma+1} \leq \|\phi\|_X$, we can write
\begin{align*}
|a'_k(r)| \leq& \underset{|\theta|=1}{\sup}|\psi_k(\theta)| \int_{S^{N-1}} |\phi_r(r, \theta)| d \theta 
\leq  \hat{C_k}\int_{|\theta|=1} |\nabla \phi(r \theta)| d\theta  \leq  \hat{C_k}\int_{|\theta|=1} \frac{\|\phi\|_X}{r^{\sigma+1}} \leq \frac{\tilde{C_k}}{r^{\sigma+1}}.
\end{align*}
From this, we deduce \begin{equation}\label{a_k second inequality}
     r^{\sigma+1} |a'_k(r)| < \Tilde{C}.
\end{equation}
The results from  \eqref{firsta_k inequality} and \eqref{a_k second inequality} give us 
$ \underset{0<r<R }{\sup}  \{ r^{\sigma} |a_k(r)| +  r^{\sigma+1} |a'_k (r)\}< \infty . $
$\hfill \Box$\\
Note that the equation \eqref{kernel L_0 equation}
 is an Euler ODE. We rewrite it as
\begin{equation*}
r^2 a''_k(r)+ r a'_k(r) \left( (N-1)  -\frac{\kappa}{\beta}  \right) - \lambda_k  a_k(r) =0
\end{equation*}
and its solution can be written as $ a_k(r)=C_k r^{\gamma _k ^{+}}+D_k r^{\gamma _k ^{-}}$
where $\gamma^{\pm}$ are given by \begin{equation}\label{gamma_general}
 \gamma_k ^{\pm}(\kappa)= \frac{-(N-2-\frac{\kappa}{\beta})}{2} \pm \frac{\sqrt{(N-2-\frac{\kappa}{\beta})^2+4 \lambda_k}}{2}. 
\end{equation} 
%i put comment on page 87..... 
\noindent 
\noindent
First, we let $\kappa= -p.$\\
 We will be looking at the kernel of $L_0^-=L_0^{-p}$.
 \noindent Thus our ODE becomes \begin{equation}\label{ode-equ-k=--l_0}
    -a''_k(r)-\frac{(N-1)a'_k}{r} + \frac{\lambda_k a_k(r)}{r^2} -\frac{p a'_k(r)}{\beta r}  =0 \quad \text{for}~~ 0<r<R
\end{equation} and the solution can be written as  $a_k(r)=C_k r^{\gamma _k ^{+}}+D_k r^{\gamma _k ^{-}}$ for some $C_k,  D_k \in \mathbb{R}$ 
where $\gamma^{\pm}$ are given by \eqref{gamma_general} where $\kappa=-p$.
We claim that   $\gamma_k^- +\sigma \leq -1$ for $k \geq 1$.  
To show this, define a function $f(\gamma)=\gamma^2 + (N-2+ \frac{p}{\beta}) \gamma- \lambda_1$.  
Noting that
$-\sigma-1=\frac{-1}{p-1}$, and by considering the values of $\beta$ and $\xi$
, we get $f(-\sigma
-1)=\frac{2p(1-\xi)}{(p-1)^2}$.
Since $\xi>1$, we have $f(-\sigma-1)<0$. Note that $f(\gamma)$ is a quadratic function with $\gamma_1^\pm$ as its roots. We showed that $f(-\sigma-1)<0$. Thus, we can conclude $\gamma_1^-\leq-\sigma-1\leq \gamma_1^+ $ and this means that $\gamma_k^-\leq -\sigma-1$ for all $k \geq 1$ by the monotonicity in $k$. We now show that $\gamma_k^+ +\sigma$ is nonzero and positive.
Note that  we have
\[\gamma_1 ^{+}(-p)= \frac{-(N-2+\frac{p}{\beta})}{2} + \frac{\sqrt{(N-2+\frac{p}{\beta})^2+4(N-1)}}{2} >0.\]
Thus by the monotonicity in $k$, we see that $\gamma_k^+ +\sigma$ is positive for all $k\geq 1$
 . \\
We first consider the case where $0<R< \infty $.
%We use a similar approach to the case $\kappa=p$.
By considering the boundary condition, to have $a_k(r)$ satisfy \eqref{a_k in L_0 statement}, we must have $a_k(r)=0$. 
For the case $R= \infty$, we note that $\gamma_k^- +\sigma $ is negative, and $ \gamma_k^+ +\sigma $ is positive and they are distinct. By sending $r$ to zero and infinity, we deduce that in order to have $a_k$ satisfy \eqref{a_k in L_0 statement}, we must have $ a_k=0$. This shows that for all $k \geq 1$ \begin{equation} \label{L0-reuslt2}
    \phi =0.
\end{equation}
The proof will be very similar for the case $\kappa=p$ and a simplified proof of this case can also be found in \cite{AghajaniA.2021Sepi} and we skip the proof here.
Thus we proved the lemma, and so we showed that for $\phi \in X_1$ if $L_0^\kappa(\phi)=0$ in $B_R \backslash \{0\}$ with $\phi=0$ on $\partial B_R$ in the case of $R$ finite where $\kappa=p$ or $\kappa=-p$, then $\phi=0$.

\end{proof} 
 \subsection{Kernel of $L_t^{\pm}$}
 
Recall that we have our linear operator $L_t^k$ defined as :
\begin{equation} \label{Lt_equation}
L_t^\kappa(\phi)= -\Delta \phi \pm \frac{p}{(t|x|^{\xi-1}+\beta )}  \frac{x \cdot \nabla \phi}{|x|^2}\end{equation}
where $\kappa$ is either equal to $p$ or $-p$.
\begin{lemma} \label{kernel_Lt}
Let $0<R \leq \infty$ , $0 \le t \leq \infty$.
When $\kappa= p$,  suppose $ \phi \in X_1$ is such that $L_t^p(\phi)=0$ in $B_R \backslash \{0\}$ with $\phi=0$ on $\partial B_R$ where $R$ is finite then $\phi=0$. When  $\kappa=-p$, suppose $ \phi \in X$ is such that $L_t^{-p}(\phi)=0$ in $B_R \backslash \{0\}$ with $\phi=0$ on the boundary of $B_R$ when $R$ is finite, then $\phi=0$ and in case of $R=\infty$,  then $ \phi$ is constant. 
\end{lemma}
\begin{proof}
 For the proof we will switch notations,  
and hence by \eqref{Lt_equation} we can write $L_t^+ = L_t^p$ and $L_t^- = L_t^{-p}$.  Suppose $R, t , \phi$ are as in the hypothesis. \\
 $\bullet$ We set $\kappa=+p.$ So we are considering the kernel of $L_t^+$. A very similar proof was given in \cite{AghajaniA.2021Sepi} so we will skip it here and we will do the case $\kappa=-p$ which is more technical.  %Further, since $L_\infty =- \Delta$ and this result is well known for the Laplacian, we suppose $0< t< \infty$.

$\hfill \Box$\\
\noindent
$\bullet$ We now set $\kappa = -p.$ In this part, we focus on kernel of $L_t^-$.
 First we set $k\geq 1$. Thus $\phi $ is in $X_1$. Suppose $0< t< \infty$. We write $\phi$ as $ \phi(x)= \sum_{k=1}^{\infty} a_k(r) \psi_k(\theta),$ and then $a_k$ should satisfy
\[ -a''_k(r)-\frac{(N-1)a'_k}{r} + \frac{\lambda_k a_k(r)}{r^2} -\frac{p a'_k(r)}{\beta r +t r^{\xi}}=0 \quad \text{for}\quad 0<r<R  \] where $a_k(R)=0$ when $R<\infty$. We should have \begin{equation}\label{a_k estimate second case k=-p}\underset{0<r<R }{\sup}  \{ r^{\sigma} |a_k(r)| +  r^{\sigma+1} |a'_k (r)|\}< \infty.
\end{equation}\noindent
  We  fix $k\geq 1$ and we set $\omega(\tau) = r^\sigma a_k(r) $ where $\tau = \ln {(r)}$.
By a computation, we find that $\omega= \omega(\tau)$ satisfies
$ \omega_{\tau \tau} + g(\tau) \omega_\tau +C_k \omega=0$ for $\tau \in (-\infty , \ln{(R)}) $
where
$g(\tau)=N-2-2\sigma + \frac{p}{\beta +te^{(\xi-1)\tau}}$ and 
$C_k(\tau)= - \lambda_k - \frac{p \sigma}{\beta + t e^{(\xi-1)\tau}}- \sigma(N-2-\sigma).$
We claim the improved decay estimate:  
$
r^\sigma |a_k(r)| \to 0 \quad \text{as} \quad r \to 0,
$
and for \( R = \infty \),
$
r^\sigma |a_k(r)| \to 0 \quad \text{as} \quad r \to \infty.
$
Assuming this, we obtain \( \omega \to 0 \) as \( \tau \to -\infty \), and for \( R = \infty \), also as \( \tau \to \infty \).
By multiplying by \(-1\) if necessary, we assume \( \omega \neq 0 \) ( and since \( \omega(-\infty) = \omega(\ln R) = 0 \) ), there exists \( \tau_0 \in (-\infty, \ln R) \) with  
$
\omega(\tau_0) = \max \omega > 0.
$
This gives \( \omega_{\tau\tau}(\tau_0) \leq 0 \) and \( \omega_\tau(\tau_0) = 0 \), leading to  
$
g(\tau_0) \omega_\tau(\tau_0) = 0.
$ From the equation we can deduce 
$\omega_{\tau \tau}(\tau_0) + C_k\omega(\tau_0) = 0 $ which means $-\omega_{\tau \tau}(\tau_0) = C_k\omega(\tau_0) \geq 0.$
From this, we see that we must have 
\[C_k(\tau_0)= - \lambda_k - \frac{p \sigma}{\beta + t e^{(\xi-1)\tau_0}}- \sigma(N-2-\sigma)\geq 0 .\]\noindent
We know $\lambda_k $ is positive for all $k\geq 1$ and it is obvious that $\frac{p \sigma}{\beta + t e^{(\xi-1)\tau}}>0$ for all $\tau \in (-\infty, \ln{(R)})$. Also note that
$N-2-\sigma =N-1-\frac{1}{p-1}.$
By considering the restrictions on $p$, we find that
$N-1-\frac{1}{p-1}>0$ for $N\geq3$.
Thus we can deduce that \[-\lambda_k<0,\quad -\frac{p \sigma}{\beta + t e^{(\xi-1)\tau}}<0, \quad \text{and} \quad - \sigma(N-2-\sigma)<0. \] So
$C_k(\tau)$ is positive for all $\tau\in(-\infty, \ln{(R)})$ which means $C_k(\tau_0)<0$. Hence, we have a contradiction and thus $\omega=0$. This gives us  $a_k=0$ for all $k\geq 1$.

We now prove the claimed decay estimate.\\
\textit{Proof of the claimed decay estimate}.  We fix $k\geq 1$ and set $a(r)=a_k(r)$, so we have 
\[ -\Delta a(r)+ \frac{\lambda_k a(r)}{r^2} - \frac{p a'(r)}{\beta r + t r^{\xi}}=0 \quad \text{in} \quad 0< r< R,\]
with $a(R)=0$. 
Suppose the claim is false.  Then there is some $r_m $ that goes to zero such that $r_m^\sigma |a(r_m)| \geq \epsilon_0 >0$. Define the rescaled function $a^m(r) \coloneqq r^\sigma_m a(r_mr)$ and note that $|a_m(1)| \geq \epsilon_0$ and $r^\sigma |a^m(r)| \leq C.$  A computation shows that\[{r_m}^{\sigma+2} \left[ \Delta a(r_mr)+\frac{\lambda_k a (r_mr)}{{(r_mr)}^2} - \frac{p a'(r_mr)}{\beta r_mr+ t {(r_mr)}^{\xi}} \right]=0 \quad~~ \text{in}~~ 0< r_mr< R,\]so we get \begin{equation}\label{a_m_of kernrel_equation}
   -\Delta a^m(r)+ \frac{\lambda_k a^m(r)}{r^2} - \frac{p (a^m)'(r)}{\beta r + t r^{\xi}r_m^{\xi-1}}=0 \quad \text{in} \quad 0< r<\frac{ R}{r_m}. 
\end{equation}
By passing to the limit, we find $a^\infty$ such that it is bounded away form zero with $r^\sigma | a^\infty (r)| + r^{\sigma+1}|(a^\infty)'(r)| \leq C$. Thus, we have 
\[ -\Delta a^\infty(r)+ \frac{\lambda_k a^\infty(r)}{r^2} + \frac{p (a^\infty)'(r)}{\beta r }=0 \quad \text{in} \quad 0< r<\infty.\] 
Let $\phi(x) = a^\infty(r) \psi_k (\theta)$. So $\phi$ is nonzero too and it is in the kernel of $L_0^-$. This is a contradiction with Lemma \ref{L_0 kernel lemma} which stated the kernel of $L_{0}^-$  is trivial.\\
In the case of $R=\infty$, we  assume there exists some $r_m  $ approaching infinity as $m $ goes to infinity.  Again we pass to the limit in \eqref{a_m_of kernrel_equation}
and we get
\[ -(a^\infty)''(r)-\frac{(N-1)(a^\infty)'}{r} + \frac{\lambda_k a^\infty(r)}{r^2} =0 \quad \text{in} \quad 0< r<\infty.\] 
This is again an Euler ODE  and the characteristic equation would be $\gamma^2 +(N-2)\gamma -\lambda_k=0$ such that 
\begin{align*}
\gamma_k^{\pm}=& \frac{-(N-2)\pm \sqrt{(N-2)^2+4\lambda_k}}{2}.
\end{align*} Thus, the solution is  $a_k(r)= C_kr^{\gamma_k^+}+D_kr^{\gamma_k^-}$. We have that $\gamma_k^-+\sigma <0< \gamma_k^+ +\sigma$. So we can deduce that when we send $r$ to zero and infinity, we must have $C_k=D_k=0$ in order to have $a_k$ in the required space. This result give us that $a_k=0$ for all $k\geq 1$.\\
$\bullet$ We now set $k=0$. We have $L_t^-=0$ thus we can write 
    \begin{equation}\label{a_k_-_zero_mode}
        -a''_k(r)-\frac{(N-1)a'_k}{r} -\frac{p a'_k(r)}{\beta r +t r^{\xi}}=0 \quad \text{for}\quad 0<r<R \end{equation}
        such that$\underset{0<r<1}{\sup} \{ r^{\sigma} |a_t(r)|+r^{\sigma+1} |a'_t(r)|\} \leq C$.
    From \eqref{a_k_-_zero_mode}, using the integrating factor $\mu_t(r)$, we get: \[\frac{d}{dr}(\mu_t(r)a'_t{(r)})= 0 \Rightarrow \mu_t a'(r)=C\]
where $C$ is a constant. 
Thus we have
\begin{align*}
    r^{\sigma+1} |a_t'(r)|=\frac{C r^{\sigma+1}}{\mu_t(r)}=Cr^{\sigma+1-N+1-\frac{p}{\beta}}\Bigg(\frac{tr^{\xi-1}+\beta}{t+\beta}\Bigg)^{\frac{p}{p-1}}. 
\end{align*}
Note that $\sigma+2-N-\frac{p}{\beta}=\frac{(\xi-1)(-1-p)}{p-1}<0$. 
 Since $\phi $ is in $X$, we require $a'_t$ to satisfy the required bounds meaning  $r^{\sigma+1}|a'_t(r)|$  should be bounded. So we can deduce that $C$ should be zero. So we have 
\[\mu_t a'_t(r)=0.\]
By considering the boundary condition $a_t(R)=0$, we obtain $a_t(r)=0$ for $r\in(0,R)$ and $a_t(r)$ is constant on the whole space $\mathbb{R}^N$. We now consider the case where $t $ approaches infinity. Thus we have the ODE
\begin{equation}\label{t_infty_case}
    -a''_k(r)-\frac{(N-1)a'_k}{r} =0
\end{equation}
where $a(1)=0$. We have the integrating factor as $\mu_t(r)= r^{(N-1)}$, thus, we have
 $\frac{d}{d\tau}\left(\mu_t(\tau)a'_t{(\tau)}\right)= 0.$
Now by solving this equation, we obtain
\begin{align*}
    a(r)=C \int_r^1 \frac{1}{s^{N-1}}ds= \frac{C}{2-N}s^{2-N} \bigg|_r^1= C\left(\frac{1}{2-N}-\frac{r^{2-N}}{2-N}\right).
\end{align*}
We know $a(r)$ needs to satisfy $r^\sigma|a_t(r)|\le C_1$ for some positive constant $C_1$.
So we should have \[r^{\sigma}\left|C\left(\frac{1}{2-N}-\frac{r^{2-N}}{2-N}\right)\right|=\left|C\left(\frac{r^{\sigma}}{2-N}-\frac{r^{\sigma+2-N}}{2-N}\right)\right|\le C_1.\]
Noting that $\sigma=\frac{2-p}{p-1}>0 $ and $\frac{N}{N-1}<p<2$, by a computation we can see that $\sigma+2-N$ is negative.
So for $a_t(r)$ to satisfy the required bound when $r$ approaches zero, we should have $C=0.$ This shows that $a_t(r)=0$ when $t $ goes to infinity and so we should have $\phi=0$.\\
\noindent
 Thus the lemma is complete and we proved if  $\kappa= p$  and $ \phi \in X_1$ is such that $L_t^p(\phi)=0$ in $B_R \backslash \{0\}$ with $\phi=0$ on $\partial B_R$ in the case of $R$ finite then $\phi=0$ and when  $\kappa=-p$ and  $ \phi \in X$ is such that $L_t^{-p}(\phi)=0$ in $B_R \backslash \{0\}$ with $\phi=0$ on $\partial B_R$ in the case of $R$ finite then $\phi=0$ and in case of $k=0 $ and $R$ infinite $\phi$ is a constant.
%we have the similar proof but from the limiting equation we get a  nonzero $\phi = a^\infty \psi_k (\theta)$ in the kernel of $L_{1,\infty}$ which again is a contrediction and proves the claim.
\end{proof}
\subsection{Linear theory of $L_t^\pm$ on $X_1$; A priori estimate}
\begin{theorem}\label{apriori_theorem} There is some positive constant $C$ such that for all positive $t_m $,  and functions $\phi_m  \in X_1$ and $f^m \in Y_1$ we have
\begin{equation} 
\left\{
\begin{array}{lr}
 L_{t_m}^\pm(\phi_m)=f^m& \text{in}~~ B_1 \backslash \{0\},\\
\phi_m =0 & \text{on}~~ \partial B_1.
\end{array}
\right.
\end{equation}
One has the estimate $\|\phi_m\|_X \leq C\|f^m\|_Y $. 
\end{theorem}
\begin{proof} If we assume the result is false, then by passing to a subsequence (without renaming) there is some $C_m >0$, $t_m \geq 0 , f^m \in Y_1$ and $\phi_m \in X_1$ with 

\begin{equation}\label{main-equation-estimate}
\left\{
\begin{array}{lr}
-\Delta \phi_m + \dfrac{\kappa}{(t_m|x|^{\xi-1}+\beta )}  \dfrac{x \cdot \nabla \phi_m}{|x|^2} =f^m & \text{in}~~ B_1 \backslash \{0\},\\
\phi_m=0, & \text{on}~~ \partial B_1,
\end{array}
\right.
\end{equation}
 and  
\[\|\phi_{m}\|_X > C_m\|f^m\|_Y.\]  
By normalizing, we get $\|\phi_{m}\|_X=1$ and $\|f^m\|_Y \rightarrow 0$.\\ We claim \begin{equation}\label{apriori_claim}
      \underset{0<|x|\leq 1}{\sup}  \{ |x|^{\sigma+1} |\nabla \phi_m|\} \rightarrow 0.
\end{equation}
We will show that proving this claim results in: $  \underset{0<|x|\leq 1}{\sup}  \{ |x|^{\sigma} |\phi_m|\} \rightarrow 0. $\\
\textit{proof of the claim.} Suppose there is some $0< | x_m| < 1$ and $\epsilon_0 >0$  such that 
\[\epsilon_0 \leq  |x_m|^{\sigma+1} |\nabla \phi_m(x_m)| \leq 1.\]
 There are two cases that should be considered. Either $|x_m|$ could be bounded away from zero or it could be approaching zero.\\
\noindent
$\bullet$ Set $\kappa=p$. 
Now we prove the first case for $L_{t_m}^+$.\\
\noindent
\textbf{Case 1}: Assume $|x_m|$ is bounded away from zero.  Define $A_k$ and $\tilde{A_k}$ for  $k \geq 2$ as
\[A_k= \Bigl\{ x \in B_1 : \frac{1}{k} <|x|<1 \Bigl\} \quad \text{and}  \quad \tilde{A_k}= \Bigl\{ x \in B_1 : \frac{1}{2k} <|x|<1 \Bigl\}. \]
Note that $A_k \subset \tilde{A_k}$. 
 We have
\[ -\Delta \phi_m =f^m- \dfrac{p}{(t_m|x|^{\xi-1}+\beta )}  \dfrac{x \cdot \nabla \phi_m}{|x|^2} \quad ~ \text{in} ~~\tilde{A_k}. \]
Set $g^m \coloneqq f^m- \dfrac{p}{(t_m|x|^{\xi-1}+\beta )}  \dfrac{x \cdot \nabla \phi_m}{|x|^2} $.
We can see that
\begin{align}\label{g_m first part inequaluty}
\bigg| \dfrac{p}{(t_m|x|^{\xi-1}+\beta )}  \dfrac{x \cdot \nabla \phi_m}{|x|^2} \bigg| \leq & \dfrac{p}{(t_m|x|^{\xi-1}+\beta )} \big| \dfrac{x \cdot \nabla \phi_m}{|x|^2}\big|  \\ \leq \nonumber & \dfrac{p}{\beta }  \dfrac{\|\phi_m\|_X}{|x|^{\sigma+2}}=\dfrac{p}{\beta }  \dfrac{1}{|x|^{\sigma+2}} \leq \frac{p}{\beta} (2k)^{2+\sigma} \coloneqq \gamma_k.
\end{align}
Also we find that,\begin{equation}\label{f_m inequality}
    \underset{x \in \tilde{A}_k}{\sup} |f^m| = \underset{x \in \tilde{A}_k}{\sup} \frac{|x|^{\sigma+2} |f^m|}{|x|^{\sigma+2} }\leq \big(\underset{x \in \tilde{A}_k}{\sup} \frac{\|f^m\|_Y}{|x|^{\sigma+2}} \big) \leq (2k)^{2+\sigma} \|f^m\|_Y = \gamma_k \|f^m\|_Y
\end{equation}
which by \eqref{g_m first part inequaluty} and 
\eqref{f_m inequality} we can see that $g_m $ is bounded in $\Tilde{A}_k$. So there exists a positive constant $C$ such that
\[\|g^m(x)\|_{L^\infty (\tilde{A}_k)}\leq C\]
meaning
\[ \| \Delta \phi_m \|_{L^\infty (\tilde{A}_k)}\leq C.\]
By elliptic regularity, we can say that for some $0< \lambda< 1$, there exists $C_1>0$ such that
\[ \| \phi_m \|_{C^{1, \lambda} (\bar{A_k})}\leq C_1. \] 
Thus $\{\phi_m\}_m$ is a bounded sequence in $C^{1, \lambda} (\bar{A_k})$  for all $k\geq 2$.  By the standard compactness argument and a diagonal argument there exists some subsequence $\{\phi_{m_i}\}_{i} \subset \{\phi_m\}_m$  and  $\phi \in  C^{1, \frac{\lambda}{2}} (\bar{A_k})$  such that

\[ \phi _{m_i} \rightarrow \phi \quad C^{1, \frac{\lambda}{2}} (\bar{A_k}). \] 
So we can write
\[ \phi _{m_i} \rightarrow \phi \quad ~  C^{1, \frac{\lambda}{2}}_{\textit{loc}} (\bar{B_1}\backslash \{0\}) .\]
Suppose $t_m$ converges to some  $t \in [0,\infty]$ and by passing the limit we see that when $t \in [0,\infty)$, $\phi$ solves 
\[-\Delta \phi + \dfrac{p}{(t|x|^{\xi-1}+\beta )}  \dfrac{x \cdot \nabla \phi}{|x|^2} =0  \quad ~ \text{in}~~ B_1 \backslash \{0\} \]
with $\phi =0$ on $\partial B_1$.  When $t \to \infty$, we get
\[-\Delta \phi  =0  \quad ~ \text{in} \ B_1 \backslash \{0\} \]
with $\phi =0$ on $\partial B_1$.
Using the completeness of $\mathbb{R}^N$, we can pass to a subsequence and so $x_m $ converges to some $x_0$ such that  $|x_0|$ is bounded away from zero. We can now pass the limit in \eqref{case1_inequality} to see that $|x_0|^{\sigma+1}|\nabla \phi(x_0)| \geq \epsilon_0$ and this means $\phi \neq 0$.

\noindent We need to show that $\phi$ is in $X_1$. So we need to show that $ \phi $ belongs to $X$ and it has no $k=0$ mode.
\noindent
We showed that for every fixed $0<|x|<1$,  we have $|x|^{\sigma+1} |\nabla \phi_{m_i}(x)|\leq 1$. So by passing to the limit, we get $|x|^{\sigma+1} |\nabla \phi|\leq 1$. Thus by integration, we can show that we have $|x|^{\sigma} | \phi(x)|\leq 1$, so we get
\[\underset {0<|x|\leq 1} {\sup} \left\{ |x|^{\sigma} | \phi(x)|+  |x|^{\sigma+1} |\nabla \phi(x)| \right\} <\infty.  \]
This means that $\phi $ is in $X$. First note that since $\phi_m$ belong to $X_1$ for all $0<r\le 1$, we can write 
\begin{equation}
    \label{zero average}\int_{|\theta|=1 }\phi_m(r\theta)d\theta=\sum_{k=1}^{\infty} a_k(r) \int_{|\theta|=1} \psi _{k,m}(\theta) . 1 d\theta=0 . 
\end{equation} This shows that $\phi_m$ has  zero average over all the sphere for radius $0<r\ \leq 1$. Now we use the convergence we obtained above, and we can write
 for all $0<r\leq 1$ we have that $\theta \mapsto \phi_m (r \theta) $ converges uniformly on $S^{N-1}$ to $\theta \mapsto \phi(r \theta)$. Thus for all $0<r\leq 1$ we have
\[0=\int_{|\theta|=1 }\phi_m(r\theta)d\theta \rightarrow \int_{|\theta|=1 }\phi(r\theta)d\theta. \]
\noindent So $\phi $ also has zero average over all the sphere for radius $0<r\ \leq 1$ and hence $\phi $ is in $ X_1$. This means that $\phi \in X_1$ is nonzero. Thus for $t \in [0,\infty)$ we have a contradiction with the results from the previous lemmas. We now show that in case of $t \to \infty$  we also get a contradiction.
We saw that when $t=\infty$, we have
\begin{equation}\label{L_tequation_t=infty}
\left\{
\begin{array}{lr}
\Delta \phi =0  &~~ \text{in} ~~ B_1 \backslash \{0\},\\
\phi=0 & \text{on}~~ \partial B_1.

\end{array}
\right.
\end{equation}
\noindent
Then we write $\phi = \sum_{k=1}^{\infty} a_k(r)\psi_k (\theta)$, and
 solving the equation gives us
\[a_k= C_k r^{\gamma_k^+} +D_k r^{\gamma_k^-}\] 
where
\[ \gamma_k ^{\pm}= \frac{-(N-2)}{2} \pm \frac{\sqrt{(N-2)^2+4 \lambda_k}}{2}.\]
In Lemma \ref{kernel_Lt}, we showed that $\gamma_k^- +\sigma<0<\gamma_k^+ +\sigma$ for all $k \geq 1$. Hence, to have $\phi$ in the required space, we must have $C_k=D_k=0$ and so $\phi=0$. It is a contradiction with $\phi$ being nonzero.
\\
\noindent
\textbf{Case 2:} In this case, we assume there is some $\{ x_m\}$ such that $|x_m| \rightarrow 0$ and $|x_m|^{\sigma+1} |\nabla \phi_m(x_m)| \geq \epsilon_0 >0$. Set $s_m=|x_m|$, so we have $s_m \rightarrow 0$. Define $z_m \coloneqq s_m^{-1}x_m$ a sequence  and note that by definition  $|z_m|=1$. Thus,
\begin{equation}\label{sm_bound}
    |z_m|^{\sigma+1}s_m^{\sigma+1}|\nabla \phi_m(s_mz_m)|=s_m^{\sigma+1}|\nabla \phi_m(s_mz_m)| \geq \epsilon_0 >0.
\end{equation}
Define $\zeta_m(z) \coloneqq s_m^\sigma \phi_m (s_mz)$ for $0<|s_m z|< 1$. By \eqref{sm_bound}  \begin{equation} \label{zeta_inequality}
    |\nabla \zeta_m (z_m) |=s_m^{\sigma+1}|\nabla \phi_m(s_mz_m)| \geq \epsilon_0
\end{equation} and also  we have the bounds
\begin{equation}\label{zeta_bounds}
    |z|^\sigma|\zeta_m (z)| \leq 1, \quad \text{and} \quad |z|^{\sigma+1} |\nabla \zeta_m (z)| \leq 1.
\end{equation}
We can write
\[- \Delta \phi_m (s_mz) +\frac{p}{\beta +t_m|s_mz|^{\xi -1 }}\frac{s_mz. \nabla \phi_m(s_mz)}{|s_mz|^2}=f^m(s_mz).\]
Using our definition, we can obtain  $\Delta \zeta_m(z)= s_m^{\sigma+2} \Delta \zeta_m(s_mz)$, and  $\nabla \zeta_m(z) =s_m^{\sigma +1}\zeta_m (s_mz)$. Thus a computation shows that \begin{equation}\label{case2_equation}
    - \Delta \zeta_m (z) +\frac{p}{\beta +t_ms_m^{\xi-1}|z|^{\xi -1 }}\frac{z. \nabla \zeta_m(z)}{|z|^2}=s_m^{\sigma+2}f^m(s_mz).
\end{equation}
Note that by setting $g^m(z)\coloneqq s_m^{\sigma+2}f^m(s_mz)$, we showed that \[L^+_{t_ms_m^{\xi-1}}(\zeta_m(z))=g^m(z) \quad \text{in}\quad B_{\frac{1}{s_m}}=E_m=\big\{z: 0<|z|<\frac{1}{s_m} \big\}\] with $\zeta_m=0$ on $\partial E_m$.  Thus
\[- \Delta \zeta_m (z) =g^m(z)- \frac{p}{\beta +t_ms_m^{\xi-1}|z|^{\xi -1 }}\frac{z. \nabla \zeta_m(z)}{|z|^2} \quad \text{in}\quad E_m.\]
We define $A_k$ and $\tilde{A_k}$ for $k \geq 2$ as 
\[A_k= \Bigl\{ z \in B_1 : \frac{1}{k} <|z|<k \Bigl\} \quad \text{and} \quad \tilde{A_k}= \Bigl\{ z \in B_1 : \frac{1}{2k} <|z|<2k \Bigl\}. \]
We have that $A_k\subset \Tilde{A_k}$ and also  $\Tilde{A_k} \subset E_m$ for $m$ big enough. Thus we obtained
\[- \Delta \zeta_m (z) =g^m(z)- \frac{p}{\beta +t_ms_m^{\xi-1}|z|^{\xi -1 }}\frac{z. \nabla \zeta_m(z)}{|z|^2} \quad \text{in}\quad \Tilde{A_k}.\]
First note that \begin{equation}   \label{g_m inequality}
     \underset{z \in \tilde{A_k}}{\sup} |g^m| = \underset{z \in \tilde{A_k}}{\sup} s_m^{\sigma+2}|f^m(s_mz)|  =\underset{z \in \tilde{A_k}}{\sup} \frac{|s_m|^{\sigma+2}|z|^{\sigma+2} |g^m|}{|z|^{\sigma+2} } \leq \big(\underset{z \in \tilde{A_k}}{\sup} \frac{\|f^m\|_Y}{|z|^{\sigma+2}} \big) \leq (2k)^{2+\sigma} \|f^m\|_Y.
\end{equation}
Also we have
\begin{align} \label{G inrqulaity}
\nonumber  &\bigg| \dfrac{p}{(t_ms_m^{\xi-1}|z|^{\xi -1 }+\beta )}  \dfrac{z \cdot \nabla \zeta_m}{|z|^2} \bigg|  \leq \nonumber \dfrac{p}{(t_ms_m^{\xi-1}|z|^{\xi -1 }+\beta )}  \dfrac{|z|^{\sigma+1}  |\nabla \zeta_m|}{|z|^{\sigma+2}} \\ \le \nonumber&\dfrac{p}{(t_ms_m^{\xi-1}|z|^{\xi -1 }+\beta )}  \dfrac{|z|^{\sigma+1}|s_m|^{\sigma+1}  |\nabla \phi_m(s_mz)|}{|z|^{\sigma+2}}  \leq  \dfrac{p}{\beta }  \dfrac{\|\phi_m\|_X}{|z|^{\sigma+2}}=\dfrac{p}{\beta }  \dfrac{1}{|z|^{\sigma+2}} \leq \frac{p}{\beta} (2k)^{2+\sigma} \coloneqq \gamma_k.
\end{align}
Set $G=g^m(z)- \frac{p}{\beta +t_ms_m^{\xi-1}|z|^{\xi -1 }}\frac{z. \nabla \zeta_m(z)}{|z|^2}$.
 By \eqref{g_m inequality}, we can show that $G $ is bounded in $\Tilde{A_k}$. Thus, there exists a $C>0$ such that
$\|G\|_{L^\infty (\tilde{A}_k)}\leq C$
and since $-\Delta \zeta_m= G$, we have
 $\| \Delta \zeta_m \|_{L^\infty (\tilde{A}_k)}\leq C.$
By the elliptic regularity, we can say that for $0< \lambda< 1$ there exists a positive constant $C_1$ such that
\[ \| \zeta_m \|_{C^{1, \lambda} (\bar{A}_k)}\leq C_1. \] 
Thus $\{\zeta_m\}_m$ is a bounded sequence in $C^{1, \lambda} (\bar{A_k})$  for all $k\geq 2$.  By standard compactness argument and a diagonal argument there exist some subsequence $\{\zeta_{m_i}\}_{i} \subset \{\zeta_m\}_m$  and  $\zeta \in  C^{1, \frac{\lambda}{2}} (\bar{A_k})$  such that
$ \zeta _{m_i} \rightarrow \zeta$ in $ C^{1, \frac{\lambda}{2}} (\bar{A_k})$  
so we get
\[ \zeta _{m_i} \rightarrow \zeta \quad ~  C^{1, \frac{\lambda}{2}}_{\textit{loc}} (\mathbb{R}^N\backslash \{0\}) .\]
Suppose that $t_m s_m^{\xi-1}$ converges to some $t \in [0,\infty]$. By passing the limit in \eqref{case2_equation} we see that when $t<\infty$, $\zeta$ solves 
\[-\Delta \zeta + \dfrac{p}{(t|z|^{\xi-1}+\beta )}  \dfrac{z \cdot \nabla \zeta}{|z|^2} =0  \quad ~ \text{in}~~ \mathbb{R}^N \backslash \{0\} \] and in the case of $t=\infty$, we get 
\begin{equation}
\Delta \zeta =0  ~~ \text{in} ~~ \mathbb{R}^N \backslash \{0\}.
\end{equation}
Using the completeness of $\mathbb{R}^N$, we can pass to a subsequence such that $z_m \rightarrow z_0$ with $z_0$ bounded away from zero  and $|z_0|=1$. We can now pass the limit in \eqref{zeta_inequality} to see that  $|\nabla \zeta(z_0)| \geq \epsilon_0$ and this means that $\zeta$ is nonzero.
We now need to show that $\zeta $ is in $X_1$. So we need to show that $ \zeta $ belongs to $X$ and it has no $k=0$ mode. We showed that for every fixed $0<|z|<1$, we have $|z|^\sigma|\zeta_m (z)| \leq 1,$ and $|z|^{\sigma+1} |\nabla \zeta_m (z)| \leq 1$, so
by passing the limit in these two expressions, we get $|z|^{\sigma+1} |\nabla \zeta(x)|\leq 1$ and  $|z|^{\sigma} |\zeta|\leq 1$ in $\mathbb{R}^N \backslash\{0\}$. Thus, $\zeta$ is in $X$. We know that  $\zeta_m$ has no $k=0$ mode so $\zeta_m \in X_1$. With a similar approach to case 1, we can show that since $\zeta_m$ is in $X_1$, it has zero average over the sphere with radius $0<r\leq 1$.  With the convergence that we have obtained, we find that $\zeta $ has also zero average over the sphere and so $\zeta$ is also in $X_1$. 
 This shows that $\zeta \in X_1$ is nonzero and it satisfies $L_t(\zeta)=0$  in $\mathbb{R}^N \backslash\{0\}$  which is a contradiction with the  kernel results we obtained before.
\noindent
%\textbf{Claim}. Let $k$ be fixed, and $0<\gamma<1$. The sequence $(\psi_m(z))_m \subset C^{1, \gamma}(A_k)$ is bounded.

\noindent
\noindent
When $t=\infty$, we have
\begin{equation}
\Delta \zeta =0  ~~ \text{in} ~~\mathbb{R}^{N}\backslash \{0\}.
\end{equation}
We can write $\zeta = \sum_{k=1}^{\infty} a_k(r)\psi_k (\theta)$.
 Solving the equation gives us
\[a_k= C_k( r^{\gamma_k^+}) +D_k( r^{\gamma_k^-})\] 
where $\gamma_k ^{\pm}= \frac{-(N-2)}{2} \pm \frac{\sqrt{(N-2)^2+4 \lambda_k}}{2}$.
In Lemma \ref{kernel_Lt}, we showed that ${\gamma_k^+ + \sigma}  $ is positive and ${\gamma_k^- +\sigma}$ is negative and they are both nonzero and distinct. Hence, to have $a_k$ in the required space, we should have $C_k=D_k=0$. So $\zeta=0$ and it is a contradiction with $\zeta$ being nonzero.
\\
%In the case of $t = \infty $ recall that we have $\psi (z)= \sum_{k=1}^{\infty} a_k(r) \psi_k(\theta)$ and \[a_k(r) = C_k r^{\gamma_k^+}+D_k r^{\gamma_k^-}\]
%and hence we have:
%\[r^\sigma a_k(r) = C_k r^{\gamma_k^+ + \sigma}+D_k r^{\gamma_k^- +\sigma}\]
%and note that ${\gamma_k^+ + \sigma} \text{and} {\gamma_k^- +\sigma}$ are both nonzero and distinct and hence the only way we can have $a_k$ in the required space would be if $C_k=D_k=0$ 
\noindent
The results from case 1 and case 2 complete the proof of the claim we made in \eqref{apriori_claim} which means we have
\[\underset {0<|x|\leq 1}{\sup} |x|^{\sigma+1} | \phi_m(x)| \rightarrow 0.\]
$\hfill \Box$\\
As we mentioned before, this result gives us
\begin{equation}\label{integrated estimatw}
    \underset{0<|x|\leq 1}{\sup}  \{ |x|^{\sigma} |\phi_m|\} \rightarrow 0.
\end{equation}
 To show this, fix $0<|x|<1$ and let $\hat{x}$ be a point on the boundary of $B_1 \backslash\{0\}$. Also, let $t_1= \frac{1}{|x|}$ such that  $g(t)= \phi_m(tx)$. Then, $g'(t)= \nabla \phi_m \cdot x$ and this gives us
 \[\left|g(t_1)-g(1)\right| \le \int_1^{t_1}|g'(t)| \ dt.\]Thus, we can write 
 \begin{align} \label{integration inequality}
     |\phi_m(\hat{x})- \phi_m(x)|= |\phi_m(x)| \le \int_1^{t_1}|\nabla\phi_m(tx)| |x| \ dt.
 \end{align}
We showed that 
$|\nabla \phi_m (z)| |z|^{\sigma+1} \le \epsilon \ \text{for} \ 0<|z|\le 1$
 and we can write  \begin{align} \label{integration second inequality}
     |\nabla \phi_m(tx)|t^{\sigma+1}|x|^{\sigma+1} \le \epsilon.
 \end{align}
 By \eqref{integration inequality} and \eqref{integration second inequality}, we find that
 \begin{align*}
     |\phi_m(x)| \leq \int_1^{t_1} \frac{\epsilon |x|}{t^{\sigma+1}|x|^{\sigma+1}} \ dt =\frac{\epsilon}{|x|^{\sigma}}\int_1^{t_1} \frac{1}{t^{\sigma+1}}\ dt =\frac{\epsilon}{|x|^{\sigma}} \big[ \frac{t^{-\sigma}}{-\sigma} \big]^{t_1}_1=\frac{\epsilon}{|x|^{\sigma} \sigma} \big[ \frac{-1}{t_1^{\sigma}}+1 \big] 
 \end{align*}
 so we have \begin{align}
     \label{integration conclusion}
     |x|^{\sigma}|\phi_m(x)| \le \frac{\epsilon}{\sigma} [1-\frac{1}{t_1^{\sigma}}] \le \frac{\epsilon}{\sigma}.
 \end{align}
When $\epsilon $ goes to zero, we get that $|\phi_m(tx)|t^{\sigma+1}|x|^{\sigma+1}$ approaches zero. By \eqref{integration conclusion} we also find that $|x|^{\sigma}|\phi_m(x)| $ goes to zero.
 Thus we have \[\underset{0<|x|\leq 1}{\sup} \{|x|^{\sigma+1}|\nabla\phi_m(x)|+|x|^{\sigma}|\phi_m(x)|\}\to  0. \]
 This is a contradiction with $\|\phi\|_X=1$ .

 $\hfill \Box$\\
$\bullet$ Set $\kappa=-p$ such that
$\{\phi _m\}_m \subset X_1$, $f^m \in Y_1$ satisfy
\begin{equation}\label{case2_equation k=-p}
\left\{
\begin{array}{lr}
-\Delta \phi_m - \dfrac{p}{(t_m|x|^{\xi-1}+\beta )}  \dfrac{x \cdot \nabla \phi_m}{|x|^2} =f^m & \text{in}\quad B_1 \backslash \{0\},\\
\phi_m=0, & \text{on}\quad \partial B_1,
\end{array}
\right.
\end{equation} and we have the estimate
\[\|\phi_{m}\|_X > C\|f^m\|_Y \]
such that \begin{equation}\label{case1 K=-p inequality}
\epsilon_0 \leq  |x_m|^{\sigma+1} |\nabla \phi_m(x_m)| \leq 1.
\end{equation}
We will skip the proof as it would be very similar to the Other case.\\
\end{proof}

For the case $\kappa=+p$ one can use a continuation argument along with Theorem \ref{apriori_theorem} and the studies that has been done on $L_0$ in \cite{AghajaniA.2018Ssoe} and \cite{AghajaniA.2021Sepi} to show the following. At this part, we skip the details on the continuation argument in the case of $ \kappa=-p$ and they will be studied later in Lemma \ref{t=0 k nonzero mode} and Lemma \ref{k=0 mode lemma} .

\begin{corollary}
There exists a positive constant $C$ such that for all $f \in Y_1$ there is some $\phi \in X_1$ such that $L_t^+(\phi)=f$ in $B_1 \backslash \{0\}$ with $\phi=0$ on $\partial B_1$ and $\|\phi \|_X \leq C\|f\|_Y$. 
\end{corollary}

To get the desired result on the full space $ X$, we need to recombine it with the result for the $k=0$ mode.

\begin{lemma}\label{zeromode_lemma}($k=0$ mode for $L_t^+$) We are considering the case where $k=0$ in \eqref{ode_k_mode_b} and for all positive $t$, $a_t(r)$ (with  dependence on $t$) solves the equation.  We are also assuming \begin{equation}\label{b_r_equation}
    \underset{0<r\le1}{\sup} r^{\sigma+2} |b(r)| \leq 1
\end{equation}
thus we have 
\begin{equation}\label{a_t_equation}
- a''_t (r)  - \frac{(N-1)a_t'(r)}{r}+ \frac{ p a_t'(r)}{\beta r+ t r^\xi}=b(r)~~~~ 0<r< 1\end{equation}
with $a_t(1)=0$. Note that $a_t(r)$ also satisfies 
\begin{equation}
    \underset{0<|x|\le1}{\sup} \left\{ r^{\sigma} |a_t(r)|+r^{\sigma+1} |a'_t(r)| \right\} \leq C .
\end{equation}
\end{lemma}

\begin{proof}
    From \eqref{a_t_equation}, using the integrating factor $\mu_t(r)$, we get $-\frac{d}{d\tau}(\mu_t(\tau)a'_t{(\tau)})= \mu_t(\tau)b(\tau).$
Noting that $\mu_t(1)=1$, we can integrate both sides and we obtain
\begin{align*}-\int_s^1 \frac{d}{d\tau}(\mu_t(\tau)a'_t(\tau))d \tau= \int_s^1 \mu_t(\tau)b(\tau) d\tau
\Longrightarrow -(\mu_t(1)a'_t(1)-\mu_t(s)a'_t(s))=&\int_s^1 \mu_t(\tau)b(\tau) d\tau. 
\end{align*}
We can now deduce that
\begin{align*}
    a'_t(s)=&\frac{1}{\mu_t(s)}\bigg(a_t'(1)+\int_s^1 \mu_t(\tau)b(\tau) d\tau \bigg) .
\end{align*}\noindent
By integrating again with respect to $s$ from $r$ to $1$, and considering $a_t(1)=0$, we get
\[a_t(r)=-\int_r^1\frac{1}{\mu_t(s)}\bigg(a_t'(1)+\int_s^1 \mu_t(\tau)b(\tau) d\tau \bigg) ds. \]
\noindent
Set $a'_t(1) = -\int_{R_t}^{1}\mu_t(\tau)b(\tau)d \tau$, where $R_t^{\xi-1}t=1$ then we have
\begin{align*}
a'_t(r)=-(-1)\frac{1}{\mu_t(s)}\Big[-\int_{R_t}^{1}\mu_t(\tau)b(\tau)d \tau+ \int_r^1 \mu_t(\tau)b(\tau) d\tau \Big]=\frac{1}{\mu_t(s)}\int_r^{R_t}\mu_t(\tau)b(\tau)d \tau =-\frac{1}{\mu_t(r)}\int_{R_t}^{r} \mu_t (\tau) b(\tau) d\tau.
\end{align*}
So we can write $a_t$ as
\[a_t(r) = \int_r^1 \bigg(\frac{1}{\mu_t(s)}\int_{R_t}^s \mu_t(\tau)b(\tau) d\tau \bigg) ds, \quad 0<r\leq 1. \]
\noindent
We now need to check that $a_t$ satisfies the bounds independent of $t$ for large $t$.\\
We should consider two cases:  \textit{(i)} $0<r<R_t$, \textit{(ii)} $R_t<r \leq 1.$\\
\noindent
For $r<R_t$, we have
\begin{align*}
r^{\sigma+1} |a'_t(r)| &\leq  r^{\sigma+2-N+\frac{p}{\beta}} \bigg( \frac{tr^{\xi-1}+\beta}{t+\beta}\bigg)^{-(\frac{p}{p-1})} \int_r^{R_t} \mu_t (\tau)| b(\tau)| \  d\tau\\ & \leq  r^{\sigma+2-N+\frac{p}{\beta}} \bigg( \frac{tr^{\xi-1}+\beta}{t+\beta}\bigg)^{-(\frac{p}{p-1})} \int_r^{R_t} \frac{\tau^{\sigma+2} }{\tau^{\sigma+2}} \mu_t(\tau)| b(\tau) | \ d\tau.
\end{align*}
Note that $-\sigma-3+N-\frac{p}{\beta}=-\xi$
% \sigma+2-N+ \frac{p}{\beta}=&\frac{2-p}{p-1}+2-N+\frac{p(\xi-1)}{p-1}= \frac{(p-1)(-2+2-N+p(N-1))}{p-1}\\ =& \frac{(p-1)(Np-p-N)}{p-1}=\xi-1 .
%\begin{align*}
%    \sigma+2-N+ \frac{p}{\beta}=&\frac{2-p}{p-1}+2-N+\frac{p(\xi-1)}{p-1} =& \frac{(p-1)(Np-p-N)}{p-1}=\xi-1 .
 %   \end{align*}
    , so by \eqref{b_r_equation}, we can write
    \begin{align*}
r^{\sigma+1} |a'_t(r)| \leq & \ r^{\xi-1} \bigg( \frac{t+\beta}{tr^{\xi-1}+\beta}\bigg)^{(\frac{p}{p-1})} \int_r^{R_t} \tau^{N-1-\frac{p}{\beta}}\bigg(\frac{t\tau^{\xi-1}+\beta}{t+\beta}\bigg)^{(\frac{p}{p-1})} \frac{|b(\tau)|\tau^{\sigma+2} }{\tau^{\sigma+2}}  d\tau \\ \leq & \  \bigg( \frac{1+\beta}{tr^{\xi-1}+\beta}\bigg)^{(\frac{p}{p-1})} \bigg( \frac{1-(\frac{R_t}{r})^{1-\xi}}{\xi-1} \bigg) \leq \ \bigg( \frac{\beta+1}{\beta}\bigg)^{\frac{p}{p-1}}\frac{1}{\xi-1}.
\end{align*}
Thus we proved
\begin{equation}\label{case r<R_t}
   r^{\sigma+1} |a'_t(r)| \leq  \bigg( \frac{\beta+1}{\beta}\bigg)^{\frac{p}{p-1}}\frac{1}{\xi-1}\quad \text{for}\quad 0<r<R. 
\end{equation}
\noindent
For $r>R_t$, there exists a constant $c_q>0$ such that we have  $(a+b)^q\leq c_q(a^q+b^q)$ for $q>1$ and by using this inequality and noting that $\xi-\frac{p(\xi-1)}{p-1}<0$, and $\xi>1$, we can write 
\begin{align*}
    r^{\sigma+1} |a'_t(r)| \leq & \frac{r^{\xi-1}}{(\beta+tr^{\xi-1})^{\frac{p}{p-1}}}\int_{R_{t}}^r \frac{(\beta+t \tau^{\xi-1})^{\frac{p}{p-1}}}{\tau^{\xi}} d\tau \leq \frac{C_1 \beta^{\frac{p}{p-1}} r^{\xi-1} R_t^{1-\xi}}{(\beta+tr^{\xi-1})^{\frac{p}{p-1}}}  + \frac{C_1 (tr^{\xi-1})^\frac{p}{p-1}}{(\beta+tr^{\xi-1})^{\frac{p}{p-1}}}
\end{align*}
where $C_1$ is a constant independent of $t$. Recall we have $ tR_t^{\xi-1}=1$, thus
\[r^{\sigma+1} |a_t'(r)| \leq \frac{C_1 \beta^{\frac{p}{p-1}} tr^{\xi-1}}{(\beta+tr^{\xi-1})^{\frac{p}{p-1}}}+\frac{C_1(tr^{\xi-1})^{\frac{p}{p-1}}}{(\beta+tr^{\xi-1})^{\frac{p}{p-1}}} \leq \frac{C_1 \beta^{\frac{p}{p-1}}}{(tr^{\xi-1})^{\frac{p}{p-1}}}+C_1.\]
Since for $r \geq R_t$ we have $tr^{\sigma-1} \geq tR_t^{\sigma-1}=1$, we can deduce that
\begin{equation}\label{case r>R_t}
    r^{\sigma+1}|a_t'(r)| \leq C_1 \beta^{\frac{p}{p-1}}+C_1 =C_1(1+\beta^{\frac{p}{p-1}})~~ \text{for}~~  r\geq R_t.
\end{equation}
Combining \eqref{case r<R_t} and \eqref{case r>R_t} gives us
\[\underset{0<r\leq 1 }{\sup} r^{\sigma+1}|a_t'(r)| \leq \max \bigg{\{ }C_1(1+\beta^{\frac{p}{p-1}}), \ \bigg( \frac{\beta+1}{\beta}\bigg)^{\frac{p}{p-1}}\frac{1}{\xi-1} \bigg{\}} .\] This shows that $a_t(r)$ satisfies the equation and is bounded independent of $t$.
%Now we set $\kappa=-p$ and we are looking at $L_t^-$ in case of $(k=0)$. Again we write down the equation, and so we have:
%%    - a''_t (r)  - \frac{(N-1)a_t'(r)}{r}- \frac{ p a_t'(r)}{\beta r+ t r^\xi}=b(r)~~~~ 0<r\le 1
%\end{equation}
%with $a_t(1)=0$.

\end{proof}

We will delay the proof of the mapping properties of $L_t^-$ for now and we complete the proof of the main linear result assuming we have the $L_t^-$ mapping properties. 

%\eqref{zeromode_lemma},
\noindent
\textbf{Completion of the proof of Proposition \ref{the_real_main_result}.} Here, assuming that we have the result from Lemma  \ref{k=0 mode lemma}, we can combine it with Theorem \ref{apriori_theorem} and Lemma \ref{zeromode_lemma} to complete the proof of Proposition \ref{the_real_main_result}. 
Let $f \in Y$ and $\phi \in X$ satisfy \eqref{eq_equation} and we write $f(x)=f_0(r)+f_1(x),$ and $ \ \phi(x)= \phi_0 (r)+\phi_1(x),$ where we have split off the $k=0$ mode from $\phi_1 \in X_1, $ and $ f_1 \in Y_1$. Then by \eqref{apriori_theorem}, we can write
\begin{align*}
    \|\phi\|_X \leq & \|\phi_0\|_X+\|\phi_1\|_X
    \leq C|f_0\|_Y+C\|f_1\|_Y.
\end{align*}
 Hence, if we can show there is some constant $D>0$ (independent of $f$) such that $\|f_0\|_Y+\|f_1\|_Y \leq D\|f_0+f_1\|_Y$ then we would be done.\\
\noindent
Given $f\in Y $, we have 
\begin{align*}
f(x)=\sum_{k=0}^{\infty} b_k(r)\psi_k(\theta)=& b_0(r)\psi_0(\theta)+\sum_{k=1}^{\infty} b_k(r)\psi_k(\theta)= b_0(r)+\sum_{k=1}^{\infty} b_k(r)\psi_k(\theta)=f_0+f_1
\end{align*}
where $\psi_0(\theta)=1$ and $b_k(r)=\frac{1}{|S^{N-1}|}\int_{S^{N-1}}f(r,\theta)\psi_k(\theta)d \theta$.\\
Noting that \[\int_{S^{N-1}=1}|f(r,\theta)|d \theta \leq \dfrac{\|f\|_Y}{r^{\sigma+2}}\left|S^{N-1}\right|, \] we can obtain
\begin{align}\label{f_0 estimate}
    \|f_0\|_Y= &\underset{0<|x|\leq 1 }{\sup} r^{\sigma+2}|b_0(r)| =\underset{0<|x|\leq 1 }{\sup} r^{\sigma+2} \big| \frac{1}{|S^{N-1}|}\int_{S^{N-1}}f(r\theta)d \theta \big|
    \leq \underset{0<|x|\leq 1 }{\sup} r^{\sigma+2}  \frac{\|f\|_Y}{r^{\sigma+2} } =\|f\|_Y .
\end{align}\noindent
We know $f_1=f-f_0$ so we can write
\begin{equation}\label{f_1 estimate}
    \|f_1\|_Y=\|f-f_0\|_Y\leq \|f\|_Y+\|f_0\|_Y\leq 2\|f\|_Y.
\end{equation}
By \eqref{f_0 estimate} and \eqref{f_1 estimate} we have
\[\|f_0\|_Y+\|f_1\|_Y \leq D\|f\|_Y=D\|f_0+f_1\|_Y\]
thus 
\begin{equation}\label{new notation estimate}
    \|\phi\|_X\leq D\|f_0+f_1\|_Y=D\|f\|_Y
\end{equation}
where $D$ is a positive constant independent of $f$. 
Now we go can back to our previous notation to show that the main result we want in Proposition \ref{the_real_main_result}. We proved that if we fix $ (f,g)$ and set $F=f+g$ and $ G=f-g$ and consider $ \zeta_i$ to be a solution of the scalar problems
\begin{equation*} \label{scal_3} \left\{
\begin{array}{lr}
-\Delta \zeta_1-p | \nabla w_t|^{p-2} \nabla w_t \cdot \nabla \zeta_1=F,& \text{in}~~ B_1 \backslash \{0\}\\
\zeta_1=0 & \text{on}~~ \partial B_1,
\end{array}
\right.\quad \text{and} \quad \left\{
\begin{array}{lr}
-\Delta \zeta_2+p | \nabla w_t|^{p-2} \nabla w_t \cdot \nabla \zeta_2=G& \text{in}~~ B_1 \backslash \{0\},\\
\zeta_2=0 & \text{on}~~ \partial B_1,
\end{array}
\right.
\end{equation*} 

% \begin{equation} \label{scal_4} \left\{
% \begin{array}{lr}
% -\Delta \zeta_2+p | \nabla w_t|^{p-2} \nabla w_t \cdot \nabla \zeta_2=G& \text{in}~~ B_1 \backslash \{0\},\\
% \zeta_2=0 & \text{on}~~ \partial B_1,
% \end{array}
% \right.
% \end{equation}  
\noindent
we can define the operators
\[ L_t^\pm(\zeta)= -\Delta \zeta \pm \frac{p}{(t|x|^{\xi-1}+\beta )}  \frac{x \cdot \nabla \zeta}{|x|^2}.\] 
 By \eqref{new notation estimate} for some positive constants $C_1$ and $C_2$, we have the estimates
\[ \|\zeta_1\|_X \leq C_1\|F\|_Y\quad \text{and }\quad \|\zeta_2\|_X \leq C_2 \|G\|_Y.\]
We know $ \phi$, and $\psi$ satisfy $ \phi+ \psi=\zeta_1$ and $ \phi - \psi= \zeta_2$ From this we saw that 
\begin{equation} \label{so_101}
\phi = \frac{\zeta_1 + \zeta_2}{2}, \quad \psi = \frac{\zeta_1 - \zeta_2}{2}.
\end{equation}
\noindent
So we get
\[ \|\phi\|_X= C_1 \|\zeta_1 +\zeta_2\|_X \leq C_1( \|\zeta_1\|+\|\zeta_2\|)\leq C( \|F\|_Y+ \|G\|_Y),\]and
\[\|\psi\|_X= C_2 \|\zeta_1 -\zeta_2\|_X \leq C_2( \|\zeta_1\|+\|\zeta_2\|)\leq C(\|F\|_Y+ \|G\|_Y).
\]
Since $F=f+g ,$ and $G=f-g$ we have
\[\|F\|_Y\leq \|f\|_Y+ \|g\|_Y \quad \text{and} \quad \|G\| \leq \|f\|_Y+ \|g\|_Y. \]
 So there is some positive constant $C$ such that for all $ t \ge 0$ and for all $ f,$ and $g $ in $Y$, there are  some $ \phi$ and $\psi $ in $X$ which satisfy  
 \begin{equation}  \label{linear_end}
\left\{
\begin{array}{lr}
-\Delta \phi-p | \nabla w_t|^{p-2} \nabla w_t \cdot \nabla \psi=f& \text{in}~~ B_1 \backslash \{0\},\\
-\Delta  \psi - p | \nabla w_t|^{p-2} \nabla w_t \cdot \nabla \phi=g& \text{in}~~ B_1 \backslash \{0\},\\
\phi=\psi=0 & \text{on}~~ \partial B_1,
\end{array}
\right.
\end{equation}
and we have the estimate
\[ \| \phi \|_X + \| \psi \|_X \le C \|f\|_Y + C \|g\|_Y.\] 
This completes the main result of our linear theory, and  Proposition \ref{the_real_main_result}.

$\hfill \Box$

%\begin{proof}. Define the new norm given by 
%\[\|\phi\|_{\hat{X}} \coloneqq \underset{0<|x|\leq 1} {\sup} \{ |x|^\sigma |\phi (x)| + |x|^{\sigma+1} |\nabla \phi (x)| + |x|^{\sigma+2}|\Delta \phi(x)|\}\]

%Let $\hat{X_1} \coloneqq \{\phi \in X_1 : \|\phi\|_{\hat{X}} < \infty \}$. Then $L_0^{\pm}: \hat{X_1} \rightarrow Y_1$ is an isomorphism. For each $0<t<\infty $ we can show that $L_t^{\pm}: \hat{X_1} \rightarrow Y_1$ is an isomorphism provided we can show that we have an estimate for all $\tau \in [0,t]$. This follows directly from the previous theorem and the fact that controlling the $X $ norm allows one to control the $\hat{X}$ norm. So for all $0<t< \infty$ there is some $C_t$ such that for all $f \in Y_1 $ there is some $\phi \in \hat{X_1}$ which solves \eqref{phi equation} and one has $\|\phi\|_{\hat{X}} \leq C_t \|f\|_Y$. We can then apply the theorem again to see that $C_t$ is bounded above independently of $t$

%\end{proof}

 We now turn to the case of $ \kappa=-p$.   Consider (\ref{ode_k_mode_b}) in the case of $ \kappa=-p$ and $t=0$ given by 
  
\begin{equation} \label{ode_**}
-a''_k(r)-\frac{(N-1)a'_k}{r} + \frac{\lambda_k a_k(r)}{r^2} - \frac{pa'_k(r)}{\beta r }=b_k(r) \quad \text{for}~~ 0<r<1 
\end{equation} with $ a_k(1)=0$. 

\begin{lemma} \label{t=0 k nonzero mode}{} For all $k \ge 0$ there is some $C_k >0$ such that for all functions $ b(r)$ with $ \sup_{0<r<1} r^{\sigma+2} |b(r)|\le1$ there is some function $a(r)$ which satisfies (\ref{ode_**}) and $ \sup_{0<r<1} r^{\sigma}| a(r)| \le C_k$.
\end{lemma} 

\begin{proof} Here we assume $k\ge 1$ and $t=0$. In Lemma \ref{k=0 mode lemma} we will prove the result for  the case where  $k$ is zero and $t$ is positive . Our result will also hold true for $t=0$.  By assuming $k\ge 1$, we have an  Euler type equation. We know that the fundamental set of solutions of homogeneous version of \eqref{ode_**} play a crucial role.  The solution of the homogeneous equation is given by $ a_k(r)=C_1r^{\gamma_k^+}+C_2r^{\gamma_k^-}$ where 
\[  \gamma_k ^{\pm}= \frac{-(N-2+\frac{p}{\beta})}{2} \pm \frac{\sqrt{(N-2+\frac{p}{\beta})^2+4 \lambda_k}}{2}.\]  We can now use variation of parameters to write out the particular solution of (\ref{ode_**}) as 
\begin{equation}
a_{k,p}(r)=u_1(r)r^{\gamma_k^+}+u_2(r)r^{\gamma_k^-}.
\end{equation}
We know that
\begin{equation}
    \label{equation *}
    u_1'(r)r^{\gamma_k^+}+u_2'(r)r^{\gamma_k^-}=0.
\end{equation}
 Thus we need to solve for $u_1$ and $u_2$. We compute $a_{k,p}, a_{k,p}'$ and $a_{k,p}''$ and plug in these values in \eqref{ode_**} and we get
 \begin{align}
     &u_1'r^{\gamma_k^+-1}+u_2'\gamma_k^-r^{\gamma_k^--1}+u_1r^{\gamma_k^+-2}\left[{\gamma_k^+}^2-\gamma_k^++\gamma_k^+\left[N-1-\frac{p}{\beta}\right]-\lambda_k \right] \\ \nonumber &+ u_2r^{\gamma_k^--2}\left[{\gamma_k^-}^2-\gamma_k^-+\gamma_k^-\left[N-1-\frac{p}{\beta}\right]-\lambda_k \right]=-b_k(r).
 \end{align}
 A computation show that
 \begin{equation}
   \left[{\gamma_k^+}^2-\gamma_k^++\gamma_k^+  \left[N-1-\frac{p}{\beta}\right]-\lambda_k \right]=\left[{\gamma_k^-}^2-\gamma_k^-+\gamma_k^-\left[N-1-\frac{p}{\beta}\right]-\lambda_k \right]=0.
 \end{equation}
 Thus we have 
 \begin{align}
u_1'\gamma_k^+r^{\gamma_k^+-1}+u_2'\gamma_k^-r^{\gamma_k^--1}=-b_k(r) \quad \text{and} \quad u_1'r^{\gamma_k^+}+u_2'r^{\gamma_k^-}=0.\end{align}
 From this two equations, we can solve for $u_1$ and $u_2$ and we get:
\begin{align}
    u_1(r)=\int_{T_1}^r \frac{b_k(\tau)d \tau}{\tau^{\gamma_k^+-1}(\gamma_k^--\gamma_k^+)} \quad \text{and} \quad u_2(r)=\int_{T_2}^r \frac{-b_k(\tau)d \tau}{\tau^{\gamma_k^--1}(\gamma_k^--\gamma_k^+)}
\end{align}
where $C_1$, $C_2$ and $T_1$ and $T_2$ are to be picked later. Thus the solution of the equation \eqref{ode_**} can be written as:
\begin{align}
    a_k(r)&=C_1r^{\gamma_k^+}+C_2r^{\gamma_k^-}+r^{\gamma_k^+}\int_{T_1}^r \frac{b_k(\tau)d \tau}{\tau^{\gamma_k^+-1}(\gamma_k^--\gamma_k^+)}+r^{\gamma_k^-}\int_{T_2}^r \frac{-b_k(\tau)d \tau}{\tau^{\gamma_k^--1}(\gamma_k^--\gamma_k^+)}.
\end{align}
By the boundary condition, we need to have $a_k(1)=0$. Thus we pick $T_1=1$, $C_2=0$, $T_2=0$ and $C_1$ such that we have \begin{equation}
    C_1+\int_{0}^1 \frac{-b_k(\tau)d \tau}{\tau^{\gamma_k^--1}(\gamma_k^--\gamma_k^+)}=0 \Longrightarrow C_1=\int_{0}^1 \frac{b_k(\tau)d \tau}{\tau^{\gamma_k^--1}(\gamma_k^--\gamma_k^+)}.
\end{equation}   We first show that $C_1$ is well defined.  Note that we showed that $\gamma_k^- +\sigma+1\le 0$ and
\begin{equation}
    \label{b_k(tau) estimate} \frac{ |b_k(\tau)|}{ \tau^{\gamma_k^--1}} \le  \frac{1}{\tau^{\gamma_k^-+\sigma+1}}.
\end{equation} Thus we have 
\begin{align}
    |C_1| \le \left|\int_{0}^1 \frac{b_k(\tau)d \tau}{\tau^{\gamma_k^--1}(\gamma_k^--\gamma_k^+)}\right| \le \frac{1}{|\gamma_k^--\gamma_k^+|}\int_{0}^1 \frac{d \tau}{\tau^{\sigma+\gamma_k^-+1}}\le \frac{1}{|\gamma_k^--\gamma_k^+|} \frac{1}{|\sigma+\gamma_k^-|}
\end{align} where this is bounded by some positive constant $C$. This shows that $C_1$ is bounded.
So we found the equation of $a_k(r)$ as following 
\begin{equation}
      a_k(r) = r^{\gamma_k^+}\int_{0}^1 \frac{b_k(\tau)d \tau}{\tau^{\gamma_k^--1}(\gamma_k^--\gamma_k^+)}+r^{\gamma_k^+}\int_{1}^r \frac{b_k(\tau)d \tau}{\tau^{\gamma_k^+-1}(\gamma_k^--\gamma_k^+)}+r^{\gamma_k^-}\int_{0}^r \frac{-b_k(\tau)d \tau}{\tau^{\gamma_k^--1}(\gamma_k^--\gamma_k^+)}.
\end{equation}
With this choice of parameters, $a_k(r)$ satisfies the equation \eqref{ode_**} and the boundary condition. We now need to show that it satisfies the estimate as well. We have
\noindent
\begin{align}
     \nonumber r^{\sigma}|a_k(r)|&\leq \left| \frac{r^{\sigma}r^{\gamma_k^+}}{\gamma_k^--\gamma_k^+}\int_{0}^1 \frac{b_k(\tau)d \tau}{\tau^{\gamma_k^--1}}+ \frac{r^{\sigma}r^{\gamma_k^+}}{\gamma_k^--\gamma_k^+}\int_r^1 \frac{b_k(\tau)d \tau}{\tau^{\gamma_k^+-1}}+\frac{r^{\sigma}r^{\gamma_k^-}}{\gamma_k^--\gamma_k^+}\int_0^r \frac{b_k(\tau)d \tau}{\tau^{\gamma_k^--1}} \right|.
\end{align}
Noting that we have the estimate \eqref{b_k(tau) estimate} and $\sigma+\gamma_k^+>0$, we see that there exists a positive constant $C$ such that 
\begin{equation}  \label{first term estimate}
  \left| \frac{r^{\sigma}r^{\gamma_k^+}}{\gamma_k^--\gamma_k^+}\int_r^1 b_k(\tau)\frac{d \tau}{\tau^{\gamma_k^+-1}}\right|\le r^{\gamma_k^++\sigma}\int_r^1 |b_k(\tau)|\frac{d \tau}{\tau^{\gamma_k^+-1}}\le \frac{r^{\gamma_k^++\sigma}}{|\sigma+\gamma_k^+|}\left(1-\frac{1}{r^{\sigma+\gamma_k^+}}\right) \le C r^{\gamma_k^++\sigma}\le C.
\end{equation} 
We now need to examine the other two terms. Similarly, we should note that we have
$\underset{0<r<1}{\sup}r^{\sigma+2}|b_k(r)| \le 1$, and $\sigma+\gamma_k^+>0$. Thus we get \begin{align} &\label{second term estimate} \left|\frac{r^{\gamma_k^++\sigma}}{\gamma_k^--\gamma_k^+}\int_{0}^1 \frac{b_k(\tau)d \tau}{\tau^{\gamma_k^--1}} + \frac{r^{\sigma}r^{\gamma_k^-}}{\gamma_k^--\gamma_k^+}\int_0^r b_k(\tau)\frac{d \tau}{\tau^{\gamma_k^--1}} \right| \le \frac{1}{|\gamma_k^--\gamma_k^+|}\left(\frac{r^{\gamma_k^++\sigma}}{|\sigma+\gamma_k^-|}  + \frac{r^{\gamma_k^-+\sigma}}{|\sigma+\gamma_k^-|r^{\gamma_k^-+\sigma}} \right) \\& \nonumber\le r^{\gamma_k^++\sigma}+\le C
\end{align}
The bounds we get from \eqref{first term estimate} and \eqref{second term estimate} give us the desired estimate on $a_k(r)$.

%We pick $C_1$ to be $\frac{1}{(\gamma_k^--\gamma_k^+)(\sigma+\gamma_k^-)} $ and we note that $\sigma+\gamma_k^+>0$. So there exists a positive constant $C$ such that 
%\begin{equation}
 %   r^{\sigma}|a_k(r)|\le C.
%\end{equation}

%\nonumber &\le \left| C_1r^{\sigma+\gamma_k^+}-C_1r^{\sigma+\gamma_k^-}+\frac{r^{\sigma}r^{\gamma_k^+}}{\gamma_k^--\gamma_k^+}\int_r^1 \frac{d \tau}{\tau^{\sigma+\gamma_k^++1}}+\frac{r^{\sigma}r^{\gamma_k^-}}{\gamma_k^--\gamma_k^+}\int_r^1 \frac{d \tau}{\tau^{\sigma+\gamma_k^-+1}} \right| \\ \nonumber &\le \left| C_1r^{\sigma+\gamma_k^+}-C_1r^{\sigma+\gamma_k^-}+ \frac{1}{\gamma_k^--\gamma_k^+}\left(\frac{r^{\sigma+\gamma_k^+}}{-\sigma-\gamma_k^+}\left(1-\frac{1}{r^{\sigma+\gamma_k^+}}\right)+\frac{r^{\sigma+\gamma_k^-}}{-\sigma-\gamma_k^-}\left(1-\frac{1}{r^{\sigma+\gamma_k^-}}\right)\right)\right|\\ \nonumber&\le \Bigg| C_1r^{\sigma+\gamma_k^+}-C_1r^{\sigma+\gamma_k^-}+ \frac{r^{\sigma+\gamma_k^+}}{(\gamma_k^--\gamma_k^+)(\sigma+\gamma_k^+)}+\frac{r^{\sigma+\gamma_k^-}}{(\gamma_k^--\gamma_k^+)(\sigma+\gamma_k^-)} \\ \nonumber& \qquad  -\frac{1}{(\gamma_k^--\gamma_k^+)(\sigma+\gamma_k^+)}-\frac{1}{(\gamma_k^--\gamma_k^+)(\sigma+\gamma_k^-)} \Bigg|\\ &\le \left| C_1r^{\sigma+\gamma_k^+}+\frac{r^{\sigma+\gamma_k^+}}{(\gamma_k^--\gamma_k^+)(\sigma+\gamma_k^+)} -\frac{1}{(\gamma_k^--\gamma_k^+)(\sigma+\gamma_k^+)}-\frac{1}{(\gamma_k^--\gamma_k^+)(\sigma+\gamma_k^-)}\right|

\end{proof}

%\begin{theorem} Let $X,Y$ be Banach spaces and $L_0 \colon X \rightarrow Y$ be a linear, continuous and surjective operator. If there exits a positive constant $C$ such that for all $ t \in [0,1]$, $y \in Y$, and $x \in X$ with $ L_t(x)=y$ , then  $\|x\|_X \le C\|y\|_Y$, then $L_1$ is surjective and for all $y \in Y$ there exists  a $x \in X$ such that $ \ L_1(x)=y$ and $\|x\|_X \le C \|y\|_Y$.
   % Let $X,Y$ be Banach spaces and $L_0 \colon X \rightarrow Y$ be linear, continuous and surjective map. If there exits a constant $C_1 >0$ such that for any $ t \in [0,1]$ and for any pairs $(x,y)$ in $X\times Y$ satisfying $ L_t(x)=y$ ,  , then we have $\|x\|_X \le C_1 \|y\|_Y$. Then $L_1$ is surjective and $\forall y \in Y~~, \exists x\in X$ such that $ \ L_1(x)=y$ and $\|x\|_X \le C_1 \|y\|_Y$.
%\end{theorem} 

\noindent 
 \textbf{Proof of Proposition \ref{main_result}.}  Recall we are trying to show there is some $C>0$ such that for all $ f \in Y$ and $ t\ge0$ there is some $ \phi \in X$ which satisfies 
   \begin{equation} 
   \label{new equation}
\left\{
\begin{array}{lr}
L_t^\pm(\phi)=f & \text{in}~~ B_1 \backslash \{0\},\\
\phi =0 & \text{on}~~ \partial B_1.
\end{array}
\right.
\end{equation}   Moreover one has $ \| \phi \|_X \le C \|f\|_Y$. 

For the case $\kappa=+p$, the result has been proved in \cite{AghajaniA.2021Sepi}. We are going to show the result is also true for the case $\kappa=-p$. First we prove the result on space $X_1$ and to get the desired result on the full space $ X$, we will  recombine it with the result for the $k=0$ mode in Lemma \ref{k=0 mode lemma}.
We fix some $0<T<\infty$, and define the set of $A$ to be all $t \in [0,T]$ such that there exists a $C_t>0$  that for all $f \in Y_1$ there exists a $\phi \in X_1$ satisfying \eqref{new equation} and the estimate \begin{equation}
     \|\phi\|_X \le C_T \|f
     \|_Y.
 \end{equation}
 In Lemma \ref{t=0 k nonzero mode}, we showed that $0 \in A$, thus the set $A$ is non-empty. We are going to show that $A$ is closed and open.\\
 First to show that is open, we let $t_0$ be in the set $A$, and we are going to show that for some small $\epsilon$ we have that $t=t_0+ \epsilon$ is also in the set $A$. This means that we need to show that there exists a $C_t>0$ such that if we let $f\in Y_1$ then there exits $\phi \in X_1$ such that \begin{equation}  
\left\{
\begin{array}{lr}
L_t^-(\phi)=f ,& \text{in}~~ B_1 \backslash \{0\},\\
\phi =0,& \text{on}~~ \partial B_1,
\end{array}
\right.
\end{equation} and  they satisfy the estimate. 
Since $t_0$ is in the set A, thus there exists a $C_{t_0}$ such that for all $f \in Y_1$ there exists a $\phi_0$ such that they satisfy \begin{equation}  
\label{phi-0 equation}
\left\{
\begin{array}{lr}
L_t^-(\phi_0)=f ,& \text{in}~~ B_1 \backslash \{0\},\\
\phi_0 =0,& \text{on}~~ \partial B_1,
\end{array}
\right.
\end{equation} and the estimate \begin{equation}
\label{phi-0 estimate}
     \|\phi_0\|_X \le C_{t_0} \|f
     \|_Y.
 \end{equation} We look for a solution of the form $\phi=\phi_0+\psi$ where $\psi$ is unknown. We let $L_t^-(\phi)= -\Delta \phi+a_t \nabla \phi$ where $a_t=\frac{-px}{(t|x|^{\xi-1}+\beta )|x|^2}$. Thus we want \begin{align*}
     -\Delta ( \phi_0+ \psi)+a_t (\nabla \phi_0 +\nabla \psi)& =f \\ \iff  -\Delta  \phi_0+ a_{t_0} \nabla \phi_0 - \Delta \psi+ [a_{t_0+\epsilon}-a_{t_0}] \nabla \phi_0 + a_{t_0+\epsilon}\nabla \psi& =f.
\end{align*}
By \eqref{phi-0 equation}, we find that
 \begin{align}\label{psi equation}
  \nonumber f - \Delta \psi+ [a_{t_0+\epsilon}-a_{t_0}] \nabla \phi_0 + a_{t_0+\epsilon}\nabla \psi& =f\\ \nonumber\iff  - \Delta \psi+ a_{t_0}\nabla 
\psi+[a_{t_0+\epsilon}-a_{t_0}] \nabla \phi_0 + [a_{t_0+\epsilon}-a_{t_0}]  \nabla \psi& =0 \\ \iff L_{t_0}^-(\psi)=[a_{t_0}-a_{t_0+\epsilon}] \nabla \phi_0 + [a_{t_0}-a_{t_0+\epsilon}]  \nabla \psi& . 
\end{align}
Thus, we need to find $\psi$ such that it satisfies \eqref{psi equation}.
%To show that it satisfies the estimate note that if we suppose that there exists $C>0$ such that for all $f \in Y$ that $\|f\|_Y=1$ there exists a $\phi \in X$ with $\|\phi\|_X \le C$ such that $L_t^-(\phi)=f$. Then by previous results we can show that there exists a $C>0$ such that for a;; $f \in Y$ there exists $\phi \in X$ such that $L_t^-(\phi)=f$ and $\|\phi\|_X \le C\|f\|_Y$. Thus if the er have the result for $\|f\|_Y$ we have the result for all $f \in Y$. 
Now let $f \in Y_1$ be nonzero and set $F\coloneqq \frac{f}{\|f\|}$ so  $\|F\|_Y =1$. By noting that $\phi_0$ satisfies \eqref{phi-0 estimate} and $\|\psi\|_X\le 1$,
%we can write that there exists a $D>0$ such that \begin{equation}
 %   \|\phi\|_X=\|\phi_0+\psi\|_X \leq \|\phi_0\|_X+\|\psi\|_X\leq C+1\coloneqq D.
%\end{equation}
%Thus $\phi$ satisfies \begin{equation}
%     \|\phi\|_X \le C_t \|f
 %    \|_Y.
 %\end{equation}
we want to show that $\phi$ is a solution of $L_t^-(\phi)=f$, so we are going to apply a fixed point argument. Define the mapping $T_\epsilon(\psi)=\hat{\psi}$ such that \begin{equation}\label{hat_psi equation}
    L_{t_0}^-(\hat{\psi})=[a_{t_0}-a_{t_0+\epsilon}] \nabla \phi_0 + [a_{t_0}-a_{t_0+\epsilon}]  \nabla \psi \coloneqq f_1.
\end{equation}
We are going to do a fixed point argument on $T_\epsilon:B_1 \rightarrow B_1$ where $B_1=\left\{\psi \in X ; \|\psi\|_X \leq 1\right\}$. We need to show that for some  $\epsilon$ there exists a small $\epsilon_0>0$ such that for all $|\epsilon|< \epsilon_0$\\$(I)$ $T_\epsilon(B_1) \subset B_1$,\\ 
$(II)$ there exists some $\gamma\in (0,1)$ such that for all $\psi_1, \psi_2 \in B_1$ we have $\|T_\epsilon(\psi_2)-T_\epsilon(\psi_1)\|_X \le \gamma\|\psi_2-\psi_1\|_X.$\\
\textit{(I) Into.} We have $L_{t_0}^-(\hat{\psi})=f_1$. Let $f_1\coloneqq K+I$ where $K \coloneqq [a_{t_0}-a_{t_0+\epsilon}] \nabla \phi_0$ and $I\coloneqq [a_{t_0}-a_{t_0+\epsilon}]  \nabla \psi$.
Thus we have \begin{equation}
    \label{f_1 inequlity}\|f_1\| \leq \|K\|_Y+\|I\|_Y.
\end{equation} We can find 
\begin{align}
    \|K\|_Y &\leq \underset{0<|x|<1}{\sup}|x|^{\sigma+2}[a_{t_0}-a_{t_0+\epsilon}] \nabla \phi_0\le \underset{0<|x|<1}{\sup}|x|^{\sigma+1}|\nabla\phi_0|\left|\frac{p}{((t_0+\epsilon)|x|^{\xi-1}+\beta )}-\frac{p}{((t_0)|x|^{\xi-1}+\beta )}\right| .
\end{align}
Since $\phi_0 \in X $, we know that $ \underset{0<|x|<1}{\sup}|x|^{\sigma+1}|\nabla\phi_0|\leq 1$, so we get
\begin{align} \label{K open estimate}
    \|K\|_Y\le \underset{0<|x|<1}{\sup}|x|^{\sigma+1}|\nabla\phi_0|\left|\frac{p\ \epsilon \  |x|^{\xi-1}}{\left[((t_0+\epsilon)|x|^{\xi-1}+\beta )\right]\left[((t_0)|x|^{\xi-1}+\beta )\right]}\right|\le \epsilon \frac{p}{\beta^2}
\end{align}
and thus as $\epsilon$ goes to zero, $\|K\|_Y$ goes to zero. For $\|I\|_Y$, similarly we have \begin{align} \label{I open estimate}
    \|I\|_Y \leq \underset{0<|x|<1}{\sup}|x|^{\sigma+2}[a_{t_0}-a_{t_0+\epsilon}] \nabla \psi \le \underset{0<|x|<1}{\sup}|x|^{\sigma+1}|\nabla\psi|\left|\frac{p\ \epsilon \  |x|^{\xi-1}}{\left[((t_0+\epsilon)|x|^{\xi-1}+\beta )\right]\left[((t_0)|x|^{\xi-1}+\beta )\right]}\right|.
\end{align}
Since $\psi \in X $, we know that $ \underset{0<|x|<1}{\sup}|x|^{\sigma+1}|\nabla\psi|\leq 1$ so we get
 \begin{align}
      \|I\|_Y\le \epsilon \frac{p}{\beta^2}
 \end{align}
 and thus as $\epsilon$ goes to zero $\|I\|_Y$ goes to zero.
By \eqref{f_1 inequlity}, we can deduce that for $\epsilon$ small we have $T_\epsilon(B_1) \subset B_1$ and so $T_\epsilon$ is into.\\
\textit{(II) Contraction.} We need to show that for some $\gamma\in (0,1)$, we have \begin{equation}
    \|T_\epsilon(\psi_2)-T_\epsilon(\psi_1)\|_X \le \gamma\|\psi_2-\psi_1\|_X.
\end{equation}
We set the right hand side of \eqref{hat_psi equation} to be $f_1$ and we can write it as $f_1 \coloneqq K(\phi_0)+I(\psi)$ where $K \coloneqq [a_{t_0}-a_{t_0+\epsilon}] \nabla \phi_0$ and $I\coloneqq [a_{t_0}-a_{t_0+\epsilon}]  \nabla \psi$. Thus we can write $\|T_\epsilon(\psi_2)-T_\epsilon(\psi_1)\|_X $ as
\begin{align}
    \nonumber  &\left\| K(\phi_0)+I(\psi_2)- K(\phi_0)-I(\psi_1)\right\|_Y =\| I(\psi_2)- I(\psi_1)\|_Y\leq \underset{0<|x|<1}{\sup}|x|^{\sigma+2}|\nabla \psi_2-\nabla\psi_1|[a_{t_0}-a_{t_0+\epsilon}] \\& \le \nonumber \ \underset{0<|x|<1}{\sup}|x|^{\sigma+1}|\nabla(\psi_2-\psi_1)|\left|\frac{p\ \epsilon \  |x|^{\xi-1}}{\left[((t_0+\epsilon)|x|^{\xi-1}+\beta )\right]\left[((t_0)|x|^{\xi-1}+\beta )\right]}\right|.
\end{align}
Similar to \eqref{I open estimate}, we can find that
\begin{equation}
    \| I(\psi_2)- I(\psi_1)\|_Y \le \|\psi_2-\psi_1\|_X \left|\frac{p\ \epsilon \  |x|^{\xi-1}}{\left[((t_0+\epsilon)|x|^{\xi-1}+\beta )\right]\left[((t_0)|x|^{\xi-1}+\beta )\right]}\right| \le \epsilon \frac{p}{\beta^2}\|\psi_2-\psi_1\|_X
\end{equation}
For $\epsilon$ small, we get that there exists some $\gamma \in (0,1)$ such that 
\begin{equation}
    \| I(\psi_2) -I(\psi_1)\|_Y\leq \gamma\|\psi_2-\psi_1\|_X
\end{equation}
which gives us 
\begin{equation}
    \|T_\epsilon(\psi_2)-T_\epsilon(\psi_1)\|_X \le \gamma\|\psi_2-\psi_1\|_X.
\end{equation}
Thus $T_\epsilon$ is a also a contraction. So we can apply Banach's Fixed point Theorem and thus there exists $\psi \in X_1$ such that it is the fixed point of $T_\epsilon$. 
So we showed that there exists a constant $C_{t}$ such that for  $F \in Y_1$ where $\|F\|=1$ 
 there exists a $\phi$ in $X_1$ such that $L_t(\phi)=F$.
We show that there exists a $C_{t_1}>0$ such that for all $f \in Y_1$ there exists $\phi \in X_1$ such that they satisfy $L_t(\phi)=f$ and the estimate. Thus, using the linearity of $L_t$, we can write
\begin{align*}
    L_t(\phi)=F=\frac{f}{\|f\|_Y} \Rightarrow \left(\|f\|_Y L_t^-(\phi)\right)=L_t^-(\|f\|_Y\phi)=f
\end{align*}
Now we set $\phi\coloneqq \|f\|_Y \phi$, so we have \begin{equation}L_t^-(\phi)=f \quad \text{and }\quad 
    \|\phi\|_X\le \|\|f\|_Y \phi\|_X \le \|\phi\|_X\|f\|_Y \le C_{t_1}\|f\|_Y.
\end{equation}
Thus for all $f \in Y$ there exists some $\phi \in X_1$ such that $L_t^-(\phi)=f$ and $\|\phi\|_X \le C_{t_1}\|f\|.$ 
 This means that $t=t_0+\epsilon$ is in the set $A$ and thus $A$ is open.\\ 
We now show that $A$ is also closed. Let $t_m$ be in $A$ such that $t_m$ converges to $t \in [0,T]$. The goal is to show that $t $ is also in $A$. So we need to show that if we have $f \in Y_1$ we can find  $\phi \in X_1$ such that $L_t(\phi)=f$ and they satisfy the estimate. Since $t_m \in A$ thus for all $f \in Y_1$ there exists  $\phi_m \in  X_1$ such that $L_{t_m}(\phi_m)=f$ and $\|\phi_m\|_X \le C_{t_m} \|f\|_Y$. From $L_{t_m}(\phi_m)=f$, we get 
\begin{align}
    -\Delta \phi_m +a_t \nabla \phi_m=f
\end{align} where $a_t=\frac{-px}{(t|x|^{\xi-1}+\beta )|x|^2}$. Thus we have \begin{equation}
    -\Delta \phi_m =- a_t \nabla \phi_m+f \coloneqq g_m.
\end{equation}
We first assume $C_{t_m}$ is bounded. Similar to before for $k \geq 2 $ we define the two sets \[A_k= \Bigl\{ x \in B_1 : \frac{1}{k} <|x|<1 \Bigl\} \quad \text{and}  \quad \tilde{A_k}= \Bigl\{ x \in B_1 : \frac{1}{2k} <|x|<1 \Bigl\}. \]
such  that $A_k \subset \tilde{A_k}$. With a similar approach as \eqref{phi_m equation} we can show that $g_m$ is bounded in $\Tilde{A}_k$ and we have $\| \Delta \phi_m \|_{L^\infty (\tilde{A}_k)}\leq C$. By elliptic regularity we get that $\| \phi_m \|_{C^{1, \lambda} (\bar{A}_k)}\leq C_1 $. Thus, we can use the compactness argument and the diagonal argument to deduce that  there exists a subseqence $\{\phi_{m_i}\}_{i} \subset \{\phi_m\}_m$  and  $\phi \in  C^{1, \frac{\lambda}{2}} (\bar{A_k})$  such that $ \phi _{m_i} \rightarrow \phi$ in $  C^{1, \frac{\lambda}{2}}_{\textit{loc}} (\bar{B_1}\backslash \{0\})$ and thus $\phi$ satisfies \begin{equation}
    -\Delta \phi - \dfrac{p}{(t|x|^{\xi-1}+\beta )}  \dfrac{x \cdot \nabla \phi}{|x|^2} =f  \quad ~ \text{in} \quad  B_1 \backslash \{0\}.
\end{equation}
Note that for fixed $0<|x|<1$ we have 
\begin{align}
    \|\phi\|_X=\underset {0<|x|\leq 1} {\sup} \left\{ |x|^{\sigma} | \phi|+ |x|^{\sigma+1} |\nabla \phi| \right\} \le \lim_{m \to \infty}\underset {0<|x|\leq 1} {\sup} \left\{ |x|^{\sigma} | \phi_m|+ |x|^{\sigma+1} |\nabla \phi_m| \right\} \le \lim_{m \to \infty} \|\phi_m\|_X.
\end{align}
Thus, since $t_m \in A$, we get
\begin{equation}
    \|\phi\|_X\leq \lim_{m \to \infty} \|\phi_m\|_X \le C_{t_m} \|f\|_Y.
\end{equation}
This give us that $\phi$ satisfies $L_t^-(\phi)=f $ and the estimate thus $t$ is in $A$. \\
We now assume $C_{t_m} $ is unbounded. Since by the assumption $C_m$ is the smallest possible constant that we have the estimate for, we can say that we can find some $\phi_m \in X_1$ and $f\in Y_1$ such that $L_{t_m}^-(\phi_m)=f$ and $\|\phi_m\|_X \le C_{t_m} \|f\|_Y$. But for $(C_{t_m}-1)$, we do not have the estimate and thus we get $\|\phi_m\|_X \geq (C_{t_m}-1) \|f\|_Y$. We first normalize and we get $\|\phi_m\|_X=1$ and $\|f\|_Y \longrightarrow0$. We also have \begin{equation}
    -\Delta \phi_m =- a_t \nabla \phi_m+f \coloneqq g_m
\end{equation} and with the same approach as above there exists a subseqence $\{\phi_{m_i}\}_{i} \subset \{\phi_m\}_m$  and  $\phi \in  C^{1, \frac{\lambda}{2}} (\bar{A_k})$  such that \begin{equation}
    -\Delta \phi - \dfrac{p}{(t|x|^{\xi-1}+\beta )}  \dfrac{x \cdot \nabla \phi}{|x|^2} =0  \quad ~ \text{in} \ B_1 \backslash \{0\}. 
\end{equation}
With a similar argument as in Theorem \ref{apriori_theorem}, we find that $\phi \in X_1$ is nonzero and it in the kernel of $L_t^-$ which a contradiction with our kernel results. Thus we can deduce that $C_{t_m}$ should be bounded. We have shown that $A$ is non-empty, open and closed meaning that for all $t \in (0, \infty)$ there exists $C_t$ such that for all $f \in Y_1$ there exists $\phi \in X_1$ such that \begin{equation}  
\left\{
\begin{array}{lr}
L_t^-(\phi)=f ,& \text{in}~~ B_1 \backslash \{0\},\\
\phi =0,& \text{on}~~ \partial B_1,
\end{array}
\right.
\end{equation} and  \begin{equation}
     \|\phi\|_X \le C_t \|f
     \|_Y.
 \end{equation} So for all $0<T<\infty$ there exists a $C^T \le \infty$ such that $|C_t| \le C^T$ for all $0 \le t \le T$. We should note that $C_t$ can not approach infinity since if we assume that $C_t \longrightarrow \infty$ as $t \longrightarrow \infty$, we get the same equation as \eqref{L_tequation_t=infty case k=-p} and we can apply the same argument to show that we get a contradiction. This shows that $C_t$ should be bounded and thus the proof is complete.

\hfill $\Box$

Now to get the desired result on the full space $ X$, we need to recombine it with the result for the $k=0$ mode.
\begin{lemma} \label{k=0 mode lemma}($k=0$ mode for $L_t^-$)  There is some positive constant $ C$ such that for all positive $ t$ and all $b(r)$ defined on $r \in(0,1)$ with infinite $\sup_{0<r<1}  r^{\sigma+2} |b(r)|$, there exists some $a_t$ which solves 
\begin{equation}\label{a_t-_equation}
- a''_t (r)  - \frac{(N-1)a_t'(r)}{r}- \frac{ p a_t'(r)}{\beta r+ t r^\xi}=b(r)~~~~ 0<r\le 1
\end{equation}
with $a_t(1)=0$. Also, there exists $C>0$ (independent of $t$) such that $a_t(r)$ satisfies 
\begin{equation} \label{neg_est}
\sup_{0<r<1} r^{\sigma+1} |a'_t(r)| \le C \sup_{0<r<1}  r^{\sigma+2} |b(r)|.
\end{equation}
\noindent
We are also assuming 
\begin{equation}\label{b_r-_equation}
    \underset{0<r\le1}{\sup} r^{\sigma+2} |b(r)| \le 1.
\end{equation}
\end{lemma}

%\begin{equation}\label{b-_r_equation}
   % \underset{0<r\le1}{\sup} r^{\sigma+2} |b(r)| \leq 1,
%\end{equation}

\begin{proof} 

\noindent  The same as before, we know  the integrating factor associated with our ODE \eqref{a_t-_equation} is given by  $\mu_t(r)=e^{P(r)}$
    where
    \begin{align*}
    P(r)= &\int_1^r \bigg(\frac{N-1}{\tau}+ \frac{p}{\beta \tau +t \tau^{\xi}} \bigg)d\tau
    = (N-1) \ln{r} + \frac{p}{\beta(1-
    \xi)} \bigg(\ln{\bigg(\frac{\beta r^{1-\xi}+t}{t+\beta}\bigg)}   \bigg).
   \end{align*} 
\noindent
   By noting that $\beta(1-\xi)=1-p$, the integrating factor is
 
    \begin{align*}
       \mu_t(r)=  e^{(N-1) \ln{r} + \frac{p}{\beta(1-
   \xi)} \bigg(\ln{\bigg(\frac{\beta r^{\xi-1}+t}{t+\beta}\bigg)}   \bigg)}
    = e^{\ln{r^{(N-1)}}}.e^{\ln{\bigg(\frac{\beta r^{1-\xi}+t}{t+\beta}\bigg)^{(\frac{p}{\beta(1-
   \xi)})}}}=r^{N-1+\frac{p}{\beta}}\bigg(\frac{tr^{\xi-1}+\beta}{t+\beta}\bigg)^{(\frac{p}{1-p})}.
    \end{align*}
\noindent
Thus from \eqref{a_t-_equation}, using the integrating factor, we get \[-\frac{d}{d\tau}(\mu_t(\tau)a'_t(\tau))= \mu_t(\tau)b(\tau).\]
By considering $\mu_t(1)=1$, we can integrate both sides and obtain
\begin{align*}-\int_s^1 \frac{d}{d\tau}(\mu_t(\tau)a'_t(\tau))d \tau= \int_s^1 \mu_t(\tau)b(\tau) d\tau \Longrightarrow
-(\mu_t(1)a'_t(1)-\mu_t(s)a'_t(s))= \int_s^1 \mu_t(\tau)b(\tau) d\tau.
\end{align*}Thus, we have
\begin{equation}\label{a'_t equation}
   a'_t(s)=\frac{1}{\mu_t(s)}\bigg(a_t'(1)+\int_s^1 \mu_t(\tau)b(\tau) d\tau \bigg).
\end{equation}
To obtain $a_t(r)$, we integrate \eqref{a'_t equation}  with respect to $s$ from $r$ to $1$ and we consider $a_t(1)=0$, so we deduce
\[a_t(r)=-\int_r^1\frac{1}{\mu_t(s)}\bigg(a_t'(1)+\int_s^1 \mu_t(\tau)b(\tau) d\tau \bigg) ds. \]
We set $a'_t(1) = -\int_{R_t}^{1}\mu_t(\tau)b(\tau)d \tau$ where $R_t^{\xi-1}t=1$ then we have
\[a'_t(r)=-\frac{1}{\mu_t(r)}\int_{R_t}^{r} \mu_t (\tau) b(\tau) d\tau.\]
So we can write $a_t$ as
\[a_t(r) = \int_r^1 \bigg(\frac{1}{\mu_t(s)}\int_{R_t}^s \mu_t(\tau)b(\tau) d\tau \bigg) ds ~~\text{for }~~~ 0<r\leq 1. \]
\noindent
Thus by considering \eqref{b_r-_equation},  we can find that
\begin{align}\label{a_t estomate}
 a_t(r) \leq & \int_r^1 \bigg(\frac{1}{\mu_t(s)}\int_{R_t}^s \frac{\mu_t(\tau) }{\tau^{\sigma+2}} \tau^{\sigma+2} |b(\tau)| d\tau \bigg) ds \\ \leq \nonumber & \int_r^1 s^{1-N-\frac{p}{\beta}} \bigg( \frac{ts^{\xi-1}+\beta}{t+\beta}\bigg)^{-(\frac{p}{1-p})} \int_{R_t}^{s} \bigg(\frac{t\tau^{\xi-1}+\beta}{t+\beta}\bigg)^{(\frac{p}{1-p})} \tau^{-\xi+2\frac{p}{\beta}} \  \|b(r)\|_Y \\\leq \nonumber & C_t \|b(r)\|_Y .
\end{align}
\noindent
So we have shown that $a_t$ satisfies the equation with an estimate, but $C$ can possibly depend on $t$. Now assume the result is false and
we suppose that there exists some positive  $ t_m $ such that $C_{t_m}>M $. Also, $a_{t_m}$, and $b_m$ satisfy \eqref{a_t-_equation} and \[\|a_{t_m}\|_X > M \|b_m\|_Y.\]  By normalizing, we can assume that \begin{equation}
    \|b_m\|_Y=\underset{0<r<1}{\sup} \{ r^{\sigma+2} |b_m(r)|\} \rightarrow 0 \quad \text{and} \quad \|a_{t_m}\|_X= \underset{0<r<1}{\sup} \{ r^{\sigma} |a_{t_m}(r)|+ r^{\sigma+1} |a'_{t_m}(r)|\}=1.
\end{equation} 
We claim \begin{equation}\label{negkappa_claim}
      \underset{0<r< 1}{\sup}  \{ r^{\sigma+1} |a'_{t_m}|\} \rightarrow 0.
\end{equation} Suppose there are some $ 0<r_m<1$ and $\epsilon>0$ such that \begin{equation}\label{k_begative_contract_inequality}
     r_m^{\sigma+1} |a_{t_m}(r_m)| \ge \epsilon.
\end{equation}     We need to consider two cases:  \\ 
\noindent
(I)  $ r_m $ is  bounded away from zero, \\ 
(II) $ r_m \searrow 0$. \\ 
\noindent
In either case, recall that $ a_m$ satisfies 
\begin{equation}\label{blowup_p}
- a''_{t_m}(r)  - \frac{(N-1)a_{t_m}'(r)}{r}- \frac{ p a_{t_m}'(r)}{\beta r+ t_m r^\xi}=b_m(r)\quad\text{for}\quad 0<r< 1
\end{equation}
with $a_{t_m}(1)=0$.

%\begin{align*}
  %  r_m^{\sigma} |a_t(r_m)| \leq &  r_m^{\sigma}\int_r^1 s^{1-N-\frac{p}{\beta}} \bigg( \frac{t_ms^{\xi-1}+\beta}{t_m+\beta}\bigg)^{-(\frac{p}{1-p})} \int_{R_t}^{s} \bigg(\frac{t_m\tau^{\xi-1}+\beta}{t+\beta}\bigg)^{(\frac{p}{1-p})} \tau^{-\xi+2\frac{p}{\beta}}  \tau^{\sigma+2} |b(\tau)| d\tau \\ \le & C_t \underset{0<r_m\le 1}{\sup } r_m^{\sigma+2} |b(r_m)|r_m^{\sigma}\int_r^1 s^{1-N-\frac{p}{\beta}} \bigg( \frac{1}{t_ms^{\xi-1}+\beta}\bigg)^{(\frac{p}{1-p})} s^{-\xi+2\frac{p}{\beta}+1} ds \\ \le & C_t \underset{0<r_m\le 1}{\sup } r_m^{\sigma+2} |b(r_m)|r_m^{\sigma} \int_r^1 s^{-\sigma-1} \le c_t \underset{0<r_m\le 1}{\sup } r_m^{\sigma+2} |b(r_m)|
%\end{align*}
%So since  $ \underset{0<r_m\le 1}{\sup } r_m^{\sigma+2} |b_m(r_m)| \rightarrow 0$ we have  $r_m^{\sigma} |a_t(r_m)| \to 0$ which is a contradiction.

\noindent
\textbf{Case (I).}
Note that for all small positive $ \E$,  we have $ \underset{\E<r<1}{\sup} |b_m(r)| \rightarrow 0$.  Fix $ \E_0>0$ small.  We have that 
 $|\frac{ p a_{t_m}'(r)}{\beta r+ t_m r^\xi}|$ is bounded by
 some positive $C$ on $\E_0<r<1$. Using the regularity theorem of elliptic PDE, we can say for all $\E_0>0$, there exist some $0<\lambda<1$ and a positive constant $C_1$ such that  \[\|a_{t_m}\|_{C^{1,\lambda}(2\E_0<r<1)}<C_1.\]
 So by the compactness argument and the diagonal argument, there exists some subsequence (without renaming) $a_{t_m}$ such that $a_{t_m} \to a$ in $C^{1,\frac{\lambda}{2}}_{loc} (0,1]$.
Now suppose $t_m $ converges to some $ t \in [0,\infty]$ and so we can pass to the limit in \eqref{blowup_p} to arrive at 

\begin{equation}\label{limit_ode}
- a''(r)  - \frac{(N-1)a'(r)}{r}- \frac{ p a'(r)}{\beta r+ t r^\xi}=0~~~~ 0<r< 1
\end{equation} with $ a(1)=0$ when $t$ is finite.  Also, in the case where $ t= \infty$ the equation is  
\[ - a''(r)  - \frac{(N-1)a'(r)}{r} =0 \quad 0<r<1\] with $ a(1)=0$.  Note  that we can  pass to a subseqence of $r_m$ such that $ r_m \rightarrow r_0 \in (0,1]$ and using the convergence we have, we get 
 $r_0^{\sigma+1} |a'(r_0)|> \E$.  The kernel results we obtained for $\kappa=-p$ in Lemma \ref{kernel_Lt} shows that this kernel is trivial and hence we have a contradiction.

%NEGAR: try and show that after passing to subsequence that $ a_{t_m} \rightarrow a$ in $C^{1,\gamma}_{loc} (0,1]$ and what ode does $a$ satisfy?  Now by the convergence we can show that $a \neq 0$ and this contradictions some kernel results we know (you need to handle the various cases of $t $ here.... you don't need to write out all the steps...you can say it in words... you did this already in the more general case of $ \phi$. \\ 

%((NEGAR...... YOU THERE??? YOU HAVE 3 VERSIONS OF THIS OPEN AND IT MIGHT BE SCREWING IT UP.. YOU PROBABLY WANT TO CLOSE 2 OF THEM))

\noindent
\textbf{Case (II).}  We first define 
\[ z_m(r):= r_m^\sigma \left\{ a_{t_m}(r_m r) - a_{t_m}(r_m) \right\}~~ \text{for} ~~ 0<r< \frac{1}{r_m}\]   where $z_m(1)=0~$.
%((WHAT BOUND DO YOU HAVE ON  $ r^{\sigma+1} |z_m'(r)|$ ????))     \\ 
\noindent We have 
\begin{equation}\label{case2_blowup}
- a''_{t_m}(rr_m)  - \frac{(N-1)a_{t_m}'(rr_m)}{rr_m}- \frac{ p a_{t_m}'(rr_m)}{\beta rr_m+ t_m (rr_m)^\xi}=b_m(rr_m)\quad 0<rr_m< 1
\end{equation}
so we get
\begin{equation*}
- {r_m^{\sigma+2}}a''_{t_m}(rr_m)  - \frac{(N-1){r_m^{\sigma+1}}a_{t_m}'(rr_m)}{r}- \frac{ p{r_m^{\sigma+1}} a_{t_m}'(rr_m)}{\beta r+ t_m (r)^\xi (r_m)^{\xi-1}}={r_m^{\sigma+2}}b_m(rr_m).
\end{equation*}
Thus
\begin{equation}\label{z_m_equation}
- z_m''(r)  - \frac{(N-1)z_{m}'(r)}{r}- \frac{ p z_{m}'(r)}{\beta r+  r^\xi t_m(r_m)^{\xi-1}}={r_m^{\sigma+2}}b_m(rr_m)=:g_m(r)\quad  0<r< \frac{1}{r_m}
\end{equation}
where $z_m(1)=0.$   Also note that we have\[\underset{0<r<\infty}{\sup}r^{\sigma+2}|g_m(r)|= \underset{0<rr_m<1}{\sup}(rr_m)^{\sigma+2} |b_{m}(r_mr)| \to 0.\]  Considering the boundary condition and \eqref{k_begative_contract_inequality}, we can write \begin{equation} \label{z_m inequlaity and bound}
     z_m(1)=0,\quad |z'_m(1)|=r_m^{\sigma+1}|a'_{t_m}(r_m)|\ge \epsilon, \quad \text{and} \quad  r^{\sigma+1} |z_m'(r)|=  (rr_m)^{\sigma+1} |a_{t_m}(r_mr)|\le 1. 
\end{equation}  
\noindent
Since $\big|\frac{ p z_{m}'(r)}{\beta r+  r^\xi t_m(r_m)^{\xi-1}}\big|$ is bounded, we can apply a similar argument as before and so we have  $\{z_m\}_m$ is a bounded sequence is $C^{1, \lambda}$ away from the origin and infinity, thus there exists a subsequnce (without renaming )$\{z_{m}\}_{m} $ such that $z_{m} $ converges to some $z$ in $C^{1, \frac{\lambda}{2}}$. 
Now we can pass to the limit in \eqref{z_m_equation}. We need to consider three cases depending on the limiting behaviour of $t_m(r_m)^{\xi-1}$:\\ 
\noindent
(i)  $ t_m(r_m)^{\xi-1} \rightarrow 0$, \\ 
(ii) $ t_m(r_m)^{\xi-1} \rightarrow t \in (0,\infty)$, \\ 
(iii) $ t_m(r_m)^{\xi-1} \rightarrow \infty$. \\
\noindent
Case (i): Assuming that  $ t_m(r_m)^{\xi-1}$ goes to zero, by passing the limit in equation \eqref{z_m_equation}, we get
\begin{equation*}
- z''(r)  - \frac{(N-1)z'(r)}{r}- \frac{ p z'(r)}{\beta r}=0~~~~  0<r< \infty
\end{equation*}
where $z(1)=0$. Since away from the origin and infinity we have the $C^{1, \lambda}$ 
convergence, we can pass the limit in \eqref{z_m inequlaity and bound} . Thus we have  \begin{equation}\label{z_result}
   |z'(1)|\ge \epsilon,
\end{equation} and \[ r^{\sigma+1} |z'(r)|\le 1. \]  Since $z(r)$ is in the required space, we can use the kernel results we obtained in the case of $k=0, $ and $t=0$. So we should have $z(r)=0$ on $0<r<\infty$. But this is a contradiction with \eqref{z_result}.\\
\noindent
Case (ii): 
In this case, we can use a similar approach as case (i) and by passing the limit in equation \eqref{z_m_equation}, we get
\begin{equation*}
- z''(r)  - \frac{(N-1)z'(r)}{r}- \frac{ p z'(r)}{\beta r+t}=0~~~~  0<r< \infty
\end{equation*} and so using the convergence we have obtained, we can pass to the limit in  \eqref{z_m inequlaity and bound}. Thus we have  \begin{equation}\label{z_result1}
   |z'(1)|\ge \epsilon,
\end{equation} and \[ r^{\sigma+1} |z'(r)|\le 1. \]  We again use the Lemma \ref{kernel_Lt} in the case $t \neq 0$ where similarly gives us the result of $z$ being zero. Thus have the same contradiction as case (i). \\
\noindent
Case (iii): When  $ t_m(r_m)^{\xi-1} $ approaches infinity, we get \begin{equation*}
- z''(r)  - \frac{(N-1)z'(r)}{r}=0~~~~  0<r< \infty
\end{equation*}
where $z(1)=0$. We can use the result from Lemma \ref{kernel_Lt} in the case of $t  $ approaching infinity
  which stated we should have $z(r)=0$. Thus again we have a contradiction with \eqref{z_result1}.\\ 
\noindent
The contradictions from these 3 cases prove that we have the estimate \eqref{a_t estomate} where $C$ is independent of $t$. This can complete the result and thus we have shown that $a_t$ satisfies the equation and we have the estimate independent of $t$.
\end{proof}
\section{The non-linear theory}
The goal in this chapter is to show that our nonlinear mapping $T_t(\phi, \psi)$ is into and a contraction. Thus by applying Banach's Fixed Point Theorem we can obtain its fixed point and complete the proof of Theorem \ref{main_theorem}. Before that, we need to show the estimates and the asymptotic results that we will need for the main proof.

\subsection{Estimates}

 \begin{lemma}  \label{lemmainequality}

%\item Suppose $p>2$.  Then there is some $C_{1,1}$ such that for all $ x \neq 0$ and all $y$ \[ \Big| |x+y|^p - |x|^p - p |x|^{p-2} x \cdot y \Big| \le C_{1,1} |x|^{p-2} |y|^2 + %C_{1,1} |y|^p.\]

 Suppose $1<p\leq 2$.  Then there exists some positive $C$ such that for $x, y, z \in\mathbb{R}^N$
\begin{equation}\label{first-inequality estimate}
    \left| |x+y|^p - |x|^p - p |x|^{p-2} x \cdot y \right| \le  C |y|^p.
\end{equation}

\begin{equation}\label{second inequlity estimate}
    \left||x+y|^p-p|x|^{p-2} x\cdot y - |x+z|^p+p|x|^{p-2} x\cdot z\right| \leq C\left(|y|^{p-1}+|z|^{p-1}\right)|y-z|.
\end{equation}

\begin{equation}\label{third inequality estimate}
    \left||x+y|^p-|x+z|^p\right| \leq C\left(|x|^{p-1}+|y|^{p-1}+|z|^{p-1}\right)|y-z|.
\end{equation}

 \end{lemma}

%\begin{align*}
  %  \big||x+y|^p-|x+z|^p\big| \leq (|y|^{p-1}+&|z|^{p-1}+|z|^{p-1})|y-z|.\\ \\   
%\end{align*}

\begin{proof}
We skip the proof here. The proofs with all the details are avaiblable at \cite{mythesis}.
\end{proof} 
\subsection{Asymptotics}
To prove Theorem \ref{main_theorem}, we will also need some asymptotics of $w_t$.
Note that we had 
\begin{align*}
    w_{t}(r)&=\int_{r}^{1} \dfrac{dy}{({ty^{\xi} + \beta y)}^{\frac{1}{p-1}}}, \quad \text{and} \quad
    w'_t(r)=0- \dfrac{1}{({tr^{\xi} + \beta r)}^{\frac{1}{p-1}}}= \dfrac{-1}{({tr^{\xi} + \beta r)}^{\frac{1}{p-1}}}.
\end{align*}
This means that for $0<r<1$, we have $ |w'_t(r)| \leq \min \{\frac{C_\beta}{r^{\frac{1}{p-1}}}, \frac{1}{t^{\frac{1}{p-1}}r^{N-1}} \} $
where $C_\beta=\frac{1}{\beta^{\frac{1}{p-1}}}.$
This gives us
$r^{\sigma+1}|w'_t(r)| \leq \min \{C_\beta, \frac{1}{t^{\frac{1}{p-1}}r^{N-2-\sigma}} \}. $
We should also note that for any positive $t$, we have
$\lim_{r \to 0}r^{\sigma+1}w'_t(r)=-C_\beta.
$
 To show this, noting that 
 $\sigma+1=\frac{1}{p-1}$ we can write that
\begin{align*}
\lim_{r \to 0}r^{\sigma+1}w'_t(r)=\lim_{r \to 0} \dfrac{-r^{\sigma+1}}{({tr^{\xi} + \beta r)}^{\frac{1}{p-1}}} = \lim_{r \to 0} \dfrac{-r^{\sigma+1}}{r^{\frac{1}{p-1}}({tr^{\xi-1} + \beta )}^{\frac{1}{p-1}}}= \lim_{r \to 0}\frac{-1}{({tr^{\xi-1} + \beta )}^{\frac{1}{p-1}}}=\frac{-1}{\beta^{\frac{1}{p-1}}}=-C_\beta.   
\end{align*}
Also, away from zero, we can see that $w_t(r)$ and $w'_t(r)$ converge uniformly to zero. 

\subsection{The Fixed Point Argument}
We defined the nonlinear mapping $T_t(\phi,\psi)= ( \hat{\phi}, \hat{\psi})$ via 
\begin{equation} \label{non-map}
\left\{
\begin{array}{lr}
-\Delta \hat \phi - p | \nabla w|^{p-2} \nabla w \cdot \nabla \hat \psi= \kappa_1(x) | \nabla w + \nabla \psi|^p + I(\psi)& \text{in}~~ B_1 \backslash \{0\},\\
-\Delta \hat \psi - p | \nabla w|^{p-2} \nabla w \cdot  \nabla \hat \phi= \kappa_2(x) | \nabla w + \nabla \phi|^p + I(\phi)& \text{in}~~ B_1 \backslash \{0\},\\
\phi=\psi=0 & \text{on}~~ \partial B_1,
\end{array}
\right.
\end{equation}
where 
\[ I(\zeta)= | \nabla w + \nabla \zeta|^p - | \nabla w|^p - p | \nabla w|^{p-2} \nabla w \cdot \nabla \zeta.\] 
%Let 
%\[J_t(\phi, \psi)=-\Delta \phi - p | \nabla w|^{p-2} \nabla w \cdot \nabla \psi\]
%thus we have
%\begin{equation*}\label{scal_11} \left\{
%\begin{array}{lr}
%J_t(\hat\phi, \hat\psi)=\kappa_1|\nabla w_t+\nabla \psi|^p+\{ |\nabla w_t+\nabla \psi|^p-|\nabla w_t|^p-p|\nabla w_t|^{p-2}\nabla w_t\nabla \psi\},& \text{in}~~ B_1 \backslash \{0\},\\
%J_t(\hat{\psi},\hat\phi)=\kappa_2|\nabla w_t+\nabla \phi|^p+\{ |\nabla w_t+\nabla \phi|^p-|\nabla w_t|^p-p|\nabla w_t|^{p-2}\nabla w_t\nabla \phi\},& \text{in}~~ B_1 \backslash \{0\},\\
%\phi=\psi=0, & \text{on}~~ \partial B_1.
%\end{array}
%\right.
%\end{equation*} 
So for $R>0$, we define the space $\mathcal{F}_R$ as
\[ \mathcal{F}_R:=\left\{ (\phi,\psi) \in X \times X:  \| \phi\|_X, \|  \psi\|_X \le R     \right\}.  \]    We will show that $T_t$ is a contraction on $\mathcal{F}_R$ for suitable $R$ and $t$. Also, note that on this space we have \begin{equation}
    \label{metric of space}
     \|( \phi,\psi)\|_{X \times X}:= \| \phi \|_X + \| \psi \|_X .
\end{equation}
\begin{lemma}\label{into_lemma} 
To show that our mapping $T_t$ is into, we need to show that for all $(\phi,\psi) \in \mathcal{F}_R$, we have $T_t(\phi,\psi) \in \mathcal{F}_R$.
We know that by \eqref{metric of space}, we have
\begin{align*}
    \|T_t(\phi, \psi)\|_{X \times  X} =\|(\hat{\phi}, \hat{\psi} )\|_{X \times X} = \|\hat{\phi}\|_X+ \|\hat{\psi}\|_X.
\end{align*}
First we set the right hand side of \eqref{non-map} to be $H_1(\psi)$ and $H_2(\phi)$ as  $H_1(\psi)=J_t(\psi) +Q_t(\psi)$ and $H_2(\phi)=I_t(\phi) +K_t(\phi)$ where \[J_t(\psi) \coloneqq |\nabla w_t+\nabla \psi|^p-|\nabla w_t|^p-p|\nabla w_t|^{p-2}\nabla w_t\nabla \psi,  \quad\text{and}\quad Q_t(\psi) \coloneqq \kappa_1|\nabla w_t+\nabla \psi|^p ,\]
\[I_t(\phi) \coloneqq |\nabla w_t+\nabla \phi|^p-|\nabla w_t|^p-p|\nabla w_t|^{p-2}\nabla w_t\nabla \phi,  \quad\text{and}\quad K_t(\phi) \coloneqq \kappa_2|\nabla w_t+\nabla \phi|^p. \] 
 To get the result we need, we should prove two statements for both $H_1(\psi)$ and $H_2(\phi)$. 
  First we state them for $H_1(\phi).$
\noindent
 \begin{enumerate}
     \item There is some positive constant $C$ such that for $R \in (0,1),$ $0< \delta <1$, $t>1$, and  $\psi \in B_R \subset X$ one has 
    \[\| \kappa_1|\nabla w_t+\nabla \psi|^p \|_Y \leq C \bigg(R^p+ \underset{|z|<\delta}{\sup}|k_1(z)|+ \frac{1}{t^{\frac{p}{p-1}}\delta^{(N-1)p-\sigma -2}}\bigg). \]
\item There is some positive constant $C$ such that for all $t>1$, $0<R<1$ and $\psi \in B_R$ one has 
    \[ \| |\nabla w_t+\nabla \psi|^p-|\nabla w_t|^p-p|\nabla w_t|^{p-2}\nabla w_t\nabla \psi \|_Y \leq CR^p.\]  
 \end{enumerate}
%\end{lemma}
\begin{proof}
Fix $R, $ and $\psi$ as in the hypothesis and $C$ would be a constant independent of these parameters. We have the estimates 
\begin{equation}
    \label{estimates of into}
(|x|^{\sigma+1}|\nabla\phi(x)|)^p \leq R^p\quad \text{and}\quad (|x|^{\sigma+1}|\nabla w_t(x)|)^p \leq C.
\end{equation}

First note that we have:
\begin{align*}
    \|Q_t(\psi)\|_Y \leq C \underset{0<|x|\leq 1}{\sup} |x|^{\sigma+2}|x|^{-(\sigma+1)p}|\kappa_1|\big((|x|^{\sigma +1}|\nabla  w_t|)^p+(|x|^{\sigma +1}| \nabla \psi |)^p \big). 
\end{align*}
We know that $\sigma +2-(\sigma+1)p=\frac{2-p}{p-1}(1-p)+2-p=0$,
thus we get
\begin{align*}
    \|Q_t(\psi)\|_Y  \leq& C \underset{0<|x|\leq 1}{\sup} |\kappa_1|(|x|^{\sigma +1}|\nabla  w_t|)^p+ \underset{0<|x|\leq 1}{\sup} |\kappa_1|(|x|^{\sigma +1}| \nabla \psi |)^p. 
\end{align*}We saw that $w_{t}(r)=\int_{r}^{1} \dfrac{dy}{({ty^{\xi} + \beta y)}^{\frac{1}{p-1}}},$ so
\begin{align*}
    \|Q_t(\psi)\|_Y  \leq & \underset{0<|x|\leq 1}{\sup} |\kappa_1|\bigg( |x|^{\sigma +1}{(t|x|^{\xi} + \beta |x|)}^{\frac{-1}{p-1}} \bigg)^p+ \underset{0<|x|\leq 1}{\sup} |\kappa_1|(|x|^{\sigma +1}| \nabla \psi |)^p 
\end{align*}
and since $\sigma+1= \frac{1}{p-1}$, we get
\begin{align*}
    \|Q_t(\psi)\|_Y  \leq & \underset{0<|x|\leq 1}{\sup} |\kappa_1| {(t|x|^{\xi-1} + \beta )}^{\frac{-p}{p-1}}+ \underset{0<|x|\leq 1}{\sup} |\kappa_1|(|x|^{\sigma +1}| \nabla \psi |)^p .
\end{align*}
Fix $0< \delta \ll 1$. We have $\kappa_1(x),$ and $ \kappa_2(x) $ are positive continuous functions such that $\kappa_1(0)=\kappa_2(0)=0$. We write the $ \underset{0<|x|\leq 1}{\sup} |\kappa_1(x)| {(t|x|^{\xi-1} + \beta )}^{\frac{-p}{p-1}} + \underset{0<|x|\leq 1}{\sup} |\kappa_1(x)|(|x|^{\sigma +1}| \nabla \phi |)^p $ as a supremum over $B_\delta $ and $\delta < |x| \leq 1$ and this gives us
\begin{equation}
\label{Q_t estimate}
    \|Q_t(\psi)\|_Y  \leq C \bigg(   R^p+\underset{B_\delta}{\sup} |\kappa_2(x)|  + \frac{1}{ t^{\frac{p}{p-1} }|\delta|^{(\xi-1)\frac{p}{p-1}}} \bigg).
\end{equation}
For $J_t(\psi)$, we apply the estimate  \eqref{first-inequality estimate} where we set $x=\nabla w_t$ and $y= \nabla \psi$. So we get
\begin{align}\label{J_t estimate}
    \|J_t(\psi)\|_Y 
    \leq   \underset{0< |x| \leq 1 }{\sup} |x|^{\sigma+2}|\nabla \psi|^p  \leq  \underset{0< |x| \leq 1 }{\sup} (|x|^{\sigma+1}|\nabla \psi|)^p \leq C R^p.
\end{align}
By \eqref{Q_t estimate} and \eqref{J_t estimate} we can deduce 
\begin{align*}
    \| H_1( \psi)\|_Y
    \leq  \nonumber C \bigg(   R^p+ \underset{B_\delta}{\sup} |\kappa_1|(x)  + \frac{1}{ t^{\frac{p}{p-1} }|\delta|^{(\xi-1)\frac{p}{p-1}}} \bigg) .
\end{align*}
\end{proof}
We now prove similar statements for $H_2(\phi)$.
 \begin{enumerate}
     \item There is some positive constant $C$ such that for $R \in (0,1),$ $0< \delta <1$, $t>1$, and  $\phi \in B_R $ one has 
    \[\| \kappa_2|\nabla w_t+\nabla \phi|^p \|_Y \leq C \bigg(R^p+ \underset{|z|<\delta}{\sup|\kappa_2(z)|}+ \frac{1}{t^{\frac{p}{p-1}}\delta^{(N-1)p-\sigma -2}}\bigg). \]
\item There is some positive constant $C$ such that for all $t>1$, $0<R<1$ and $\phi \in B_R$ one has 
    \[ \| |\nabla w_t+\nabla \phi|^p-|\nabla w_t|^p-p|\nabla w_t|^{p-2}\nabla w_t\nabla \phi \|_Y \leq CR^p.\]  
 \end{enumerate}
\end{lemma}
\begin{proof}
Fix $R,$ and $\phi$ as in the hypothesis and $C$ would be a constant independent of these parameters and we again use the estimates $(|x|^{\sigma+1}|\nabla\psi(x)|)^p \leq R^p$, and $(|x|^{\sigma+1}|\nabla w_t(r)|)^p \leq C.$
First note that we have:
\begin{align*}
    \|k_t(\phi)\|_Y  \leq C \underset{0<|x|\leq 1}{\sup} |x|^{\sigma+2}|x|^{-(\sigma+1)p}|\kappa_2|\big((|x|^{\sigma +1}|\nabla  w_t|)^p+(|x|^{\sigma +1}| \nabla \phi |)^p \big). 
\end{align*}
Since $\sigma+1= \frac{1}{p-1}$ and by definition of $\omega_t$ and $\sigma +2-(\sigma+1)p=0$,
thus we get:
\begin{align*}
    \|k_t(\phi)\|_Y  \leq & \underset{0<|x|\leq 1}{\sup} |\kappa_2| {(t|x|^{\xi-1} + \beta )}^{\frac{-p}{p-1}}+ \underset{0<|x|\leq 1}{\sup} |\kappa_2|(|x|^{\sigma +1}| \nabla \phi |)^p .
\end{align*}
\noindent
Similarly, fix $0< \delta \ll 1$. We know that $\kappa_1(x), \kappa_2(x) >0$ are continuous and $\kappa_1(0)=\kappa_2(0)=0$. We write the $ \underset{0<|x|\leq 1}{\sup} |\kappa_2(x)| {(t|x|^{\xi-1} + \beta )}^{\frac{-p}{p-1}} + \underset{0<|x|\leq 1}{\sup} |\kappa_2(x)|(|x|^{\sigma +1}| \nabla \phi |)^p $ as a supremum over $B_\delta $ and $\delta < |x| \leq 1$. Noting that $(\xi-1)\frac{p}{p-1}>0$, we get
\begin{align}\label{K_t estimate}
    \|k_t(\phi)\|_Y  \leq C \bigg(   R^p+\underset{B_\delta}{\sup} |\kappa_2(x)|  + \frac{1}{ t^{\frac{p}{p-1} }|\delta|^{(\xi-1)\frac{p}{p-1}}} \bigg).
\end{align}
For $I_t(\phi)$, we can apply \eqref{first-inequality estimate}, where we set $x=\nabla w_t$ and $y= \nabla \phi$. So we get:
\begin{align}\label{I_t estimate}
    \|I_t(\phi)\|_Y  \leq   \nonumber \underset{0< |x| \leq 1 }{\sup} |x|^{\sigma+2}|\nabla \phi|^p  \leq  \underset{0< |x| \leq 1 }{\sup} (|x|^{\sigma+1}|\nabla \phi|)^p \leq C R^p.
\end{align}
By \eqref{K_t estimate}, we obtain 
\begin{align*}
    \| H_2( \phi)\|_Y 
    \leq   C \bigg(   R^p+ \underset{B_\delta}{\sup} |\kappa_2|(x)  + \frac{1}{ t^{\frac{p}{p-1} }|\delta|^{(\xi-1)\frac{p}{p-1}}} \bigg) .
\end{align*}\end{proof}

\begin{lemma}\label{contraction_lemma}  
    To show our mapping is a contraction, we are going to show that if we have $0<R<1, \ 0<\delta<1, \ t>1,$ and $\phi_i, \psi_i \in B_R$ and we set $T_t(\phi_i, \psi_i)=(\hat{\phi_i},\hat{\psi_i})$,  there exists some $C>0$ such that 
    \[\| T_t(\phi_2,\psi_2)-T_t(\phi_1,\psi_1)\|_{X \times X} \leq C \big{\{} R^{p-1} +\underset{B_\delta}{\sup}\{|\kappa_1| +|\kappa_2|\}+ \frac{1}{t \delta^{\xi-1}} \big{\}} \|({\phi_2}, {\psi_2}) - ({\phi_1}, {\psi_1})\|_{X \times X}.\]
    \end{lemma}
\begin{proof}
 First note that we have:
    \begin{align*}
        \|T_t(\phi_2, \psi_2)-T_t(\phi_1, \psi_1)\|_{X \times X} = &\|(\hat{\phi}_2, \hat{\psi}_2)- (\hat{\phi}_1, \hat{\psi}_1)\|_{X \times X}  =  \|[(\hat{\phi}_2- \hat{\phi}_1), (\hat{\psi}_2- \hat{\psi}_1)]\|_{X \times X}\\ &= \| \hat{\phi }_1- \hat{\phi}_2\|_X + \| \hat{\psi }_1 -\hat{\psi}_2\|_X.
    \end{align*}
   We set the right hand sides of \eqref{non-map} to be $H_1(\psi)$ and $H_2(\phi)$ where  
 \[H_1(\psi) \coloneqq |\nabla w_t+\nabla \psi|^p-|\nabla w_t|^p-p|\nabla w_t|^{p-2}\nabla w_t\nabla \psi +\kappa_1|\nabla w_t+\nabla \psi|^p, \]
\[H_2(\phi) \coloneqq |\nabla w_t+\nabla \phi|^p-|\nabla w_t|^p-p|\nabla w_t|^{p-2}\nabla w_t\nabla \phi+ \kappa_2|\nabla w_t+\nabla \phi|^p. \] For $H_1(\psi)$ let us define $I_1$ and $K_1$ such that  \[ H_1(\psi_2)-H_1(\psi_1)= I_1 +k_1\]
    where
    \begin{equation}
    \label{I_1 equation}
    I_1(\psi_1, \psi_2) \coloneqq  |\nabla w_t+\nabla \psi_2|^p-p|\nabla w_t|^{p-2}\nabla w_t\nabla \psi_2 - |\nabla w_t+\nabla \psi_1|^p+p|\nabla w_t|^{p-2}\nabla w_t\nabla \psi_1 
\end{equation}and 
\begin{equation}
    \label{K_1 equation}
    k_1(\psi_1, \psi_2)  \coloneqq \kappa_1(x) (|\nabla w_t+\nabla \psi_2|^p- |\nabla w_t+\nabla \psi_1|^p).
\end{equation}
Similarly, we write $H_2(\phi)$ as
    \[ H_2(\phi_2)-H_2(\phi_1)= J_2 +Q_2\]
    where
    \[I_2(\phi_1, \phi_2) \coloneqq  |\nabla w_t+\nabla \phi_2|^p-p|\nabla w_t|^{p-2}\nabla w_t\nabla \phi_2 - |\nabla w_t+\nabla \phi_1|^p+p|\nabla w_t|^{p-2}\nabla w_t\nabla \phi_1 \]
    \[ Q_2(\phi_1, \phi_2)  \coloneqq \kappa_2(x) (|\nabla w_t+\nabla \phi_2|^p- |\nabla w_t+\nabla \phi_1|^p). \]
    We prove two claims for both $H_1$ and $H_2$ to show that $T_t$ is a contraction. First we state and prove our claims for
 $H_1(\psi)$.\\
  \textbf{Claim 1:} There exists a positive constant  $\hat C$ such that for $0<R<1$, $t \geq 0$ and $\phi_i, \psi_i \in B_R$  we have
\[\underset{0<|x|\leq 1 }{\sup} |x|^{\sigma+2}|k_1(\psi_1, \psi_2) | \leq  \hat{C}\ \big( \underset{B_\delta}{\sup} \ \kappa_1(x)+ \frac{1}{t\delta^{\xi-1}}+2R^{p-1}  \big) \|({\phi_2}, {\psi_2}) - ({\phi_1}, {\psi_1})\|_{X \times X}. \]
\textbf{Claim 2. }
There exists a positive constant $\Tilde{C}$ such that for $0<R<1$, $t>1$, and $\phi_i , \psi_i\in B_R$ we have:
\[\underset{0<|x| \leq 1 }{ \sup} |x|^{\sigma+2} |I_1(\psi_1, \psi_2) | \leq  \Tilde{C} 2 R^{p-1}   \|({\phi_2}, {\psi_2}) - ({\phi_1}, {\psi_1})\|_{X \times X}. \]
 \textit{Proof of claim 1.} (We will skip some details of the following proof. The proof along with all the details are available at \cite{mythesis})
We use \eqref{integral_solution} and \eqref{K_1 equation} and apply \eqref{third inequality estimate} where we take $x=\nabla w_t$, $y= \nabla \psi_2$ and $z= \nabla \psi_1$  Also we note that  $(p-1)(\sigma+1)=1$, thus we have:\begin{align*}
&\underset{0<|x|\leq 1 }{\sup} |x|^{\sigma+2}|k_1(\psi_1, \psi_2) | \leq  \underset{0<|x|\leq 1 }{\sup} |x|^{\sigma+2} \kappa_1(x) \big((|\nabla w_t+\nabla \psi_2|^p- |\nabla w_t+\nabla \psi_1|^p) \big)\\ &
   \leq   C_1  \underset{0<|x|\leq 1}{\sup}  \bigg(
    \frac{\kappa_1(x)}{t|x|^{\xi-1} +\beta} + 2C_1 R^{p-1}  \kappa_1(x)  \bigg)\|\psi_2 -\psi_1\|_X. 
\end{align*}

%\[  \| \psi_2 - \psi_1 \|_X \le \| ( \phi_1, \psi_1) - (\phi_2, \psi_2)\|_{X \times X} \] 
\noindent
Let $0< \delta <1$, so we find that:
\begin{align}
\label{delta estimate}\underset{0<|x|\leq 1}{\sup} 
   \frac{\kappa_1(x) }{t|x|^{\xi-1} +\beta} \leq   \underset{0<|x|\leq \delta}{\sup} 
   \frac{\kappa_1(x) }{t|x|^{\xi-1} +\beta}  +  \underset{\delta<|x|\leq 1}{\sup} 
    \frac{\kappa_1(x) }{t|x|^{\xi-1} +\beta}  
    \leq  \ C \underset{B_\delta}{\sup} \ \kappa_1(x)+ \frac{C}{t\delta^{\xi-1}}.
\end{align}
    Thus, there exists a positive constant $\hat{C}$ such that
\[\underset{0<|x|\leq 1 }{\sup} |x|^{\sigma+2}|k_1(\psi_1, \psi_2) | \leq \hat{C}\ \big( \underset{B_\delta}{\sup} \ \kappa_1(x)+2R^{p-1}+ \frac{1}{t\delta^{\xi-1}}  \big)\|\psi_2 -\psi_1\|_X.   \]
We should note that 
\begin{equation} \label{contraction_norm}
    \|\psi_2-\psi_1\|_X \leq \|({\phi_2}, {\psi_2}) - ({\phi_1}, {\psi_1})\|_{X \times X}= \|\psi_2-\psi_1\|_X + \|\phi_2-\phi_1\|_X
\end{equation}
which means
\begin{equation}
    \label{K_1 result} \underset{0<|x|\leq 1 }{\sup} |x|^{\sigma+2}|k_1(\psi_1, \psi_2)|  \leq \hat{C}\ \big( \underset{B_\delta}{\sup} \ \kappa_1(x)+2R^{p-1}+ \frac{1}{t\delta^{\xi-1}} \big) \|({\phi_2}, {\psi_2}) - ({\phi_1}, {\psi_1})\|_{X \times X} .
\end{equation}

$\hfill \Box$\\
\textit{Proof of claim 2.}
We apply \eqref{third inequality estimate} where we set $\nabla w_t = x, \ \nabla \psi_2 =y, \ \nabla\psi_1 = z$, thus we have
\begin{align*}
     \underset{0<|x| \leq 1 }{ \sup} |x|^{\sigma+2} |I_1(\psi_1, \psi_2) |  \leq  C \underset{0<|x|\leq 1}{\sup}  \big( (|x|^{(\sigma+1)}|\nabla \psi_2|^{p-1})+ (|x|^{(\sigma-1)}|\nabla \psi_1|^{p-1}) \big) \|\psi_2 - \psi_1\|_X  \leq  \Tilde{C} 2 R^{p-1}   \|\psi_2 - \psi_1\|_X .
\end{align*} 
Thus by \eqref{contraction_norm}, we can write
\begin{equation}
    \label{I_1 result} \underset{0<|x| \leq 1 }{ \sup} |x|^{\sigma+2} |I_1(\psi_1, \psi_2) |  \leq  \Tilde{C} 2 R^{p-1}\|({\phi_2}, {\psi_2}) - ({\phi_1}, {\psi_1})\|_{X \times X}  .  
\end{equation}
By \eqref{K_1 result} and \eqref{I_1 result}, we can deduce that
\begin{equation}\label{H_1 result}
    \|\hat{\psi}_2-\hat{\psi}_1\|_X \leq  \hat{C}\ \big( \underset{B_\delta}{\sup} \ \kappa_1(x)+ \frac{1}{t\delta^{\xi-1}}+2R^{p-1} \big) \|({\phi_2}, {\psi_2}) - ({\phi_1}, {\psi_1})\|_{X \times X}.
\end{equation}

$\hfill \Box$\\
 We now state similar claims for $H_2(\phi)$. \\
     \textbf{Claim 3:} There exists a positive constant  $\hat C$ such that for $0<R<1$, $t \geq 0$ and $\phi_i, \psi_i \in B_R$ 
\[\underset{0<|x|\leq 1 }{\sup} |x|^{\sigma+2}|Q_2(\phi_1, \phi_2) | \leq  \hat{C}\ \big( \underset{B_\delta}{\sup} \ \kappa_2(x)+2R^{p-1}+ \frac{1}{t\delta^{\xi-1}}  \big)\|(\phi_2, \psi_2) -(\phi_1, \psi_1)\|_{X \times X}. \]
\noindent
 \textbf{Claim 4:} There exists $\Tilde{C}>0$ such that for $0<R<1$, $t>1$, and $\phi_i , \psi_i\in B_R$ we have:
\[\underset{0<|x| \leq 1 }{ \sup} |x|^{\sigma+2} |J_2(\phi_2,\phi_2) | \leq  \Tilde{C} 2 R^{p-1}  \|(\phi_2, \psi_2) -(\phi_1, \psi_1)\|_{X \times X}. \]
\noindent
\textit{Proof of claim 3.}
With a similar proof to \eqref{K_1 result}, we can apply \eqref{third inequality estimate} where we set $x=\nabla w_t$, $y= \nabla \psi_2$ and $z= \nabla \psi_1$ and we have:
 \begin{align*}&\underset{0<|x|\leq 1 }{\sup} |x|^{\sigma+2}|Q_1(\phi_1, \phi_2) | \leq  \underset{0<|x|\leq 1 }{\sup} |x|^{\sigma+2} \kappa_2(x) \big((|\nabla w_t+\nabla \phi_2|^p- |\nabla w_t+\nabla \phi_1|^p) \big)\\ \leq &  C_1  \underset{0<|x|\leq 1}{\sup}  \bigg(
    \frac{\kappa_2(x)}{t|x|^{\xi-1} +\beta} + 2C_1 R^{p-1}  \kappa_2(x)  \bigg)\|\phi_2 -\phi_1\|_X. 
 \end{align*}
Similar to \eqref{delta estimate}, by letting $0< \delta <1$, we have:
\begin{align*}
\underset{0<|x|\leq 1}{\sup} 
   \frac{\kappa_2(x) }{t|x|^{\xi-1} +\beta} \leq & \ C \ \underset{B_\delta}{\sup} \ \kappa_2(x)+ \frac{C}{t\delta^{\xi-1}}.
\end{align*}
 Thus,  by \eqref{contraction_norm} there exists a positive constant $\hat{C}$ such that
\begin{equation}
    \label{Q_1 result} \underset{0<|x|\leq 1 }{\sup} |x|^{\sigma+2}|Q_1(\phi_1, \phi_2)|  \leq \hat{C}\ \big( \underset{B_\delta}{\sup} \ \kappa_2(x)+2R^{p-1}+ \frac{1}{t\delta^{\xi-1}} \big) \|({\phi_2}, {\psi_2}) - ({\phi_1}, {\psi_1})\|_{X \times X} .
\end{equation}
\noindent \textit{Proof of claim 4} Similar to the proof of \eqref{I_1 result}, we can apply \eqref{second inequlity estimate} where we take  $\nabla w_t = x, \ \nabla \psi_2 =y, \ \nabla\psi_1 = z$, and thus we have
\begin{align*}
     \underset{0<|x| \leq 1 }{ \sup} |x|^{\sigma+2} |J_1(\phi_1, \psi_2) |  & \leq  C \underset{0<|x|\leq 1}{\sup}  \big( (|x|^{(\sigma+1)}|\nabla \phi_2|^{p-1})+ (|x|^{(\sigma-1)}|\nabla \phi_1|^{p-1}) \big) \|\phi_2 - \phi_1\|_X  \leq  \Tilde{C} 2 R^{p-1}   \|\phi_2 - \phi_1\|_X .
\end{align*} 
So by \eqref{contraction_norm}, we can write
\begin{equation}
    \label{J_1 result} \underset{0<|x| \leq 1 }{ \sup} |x|^{\sigma+2} |J_1(\phi_1, \phi_2) |  \leq  \Tilde{C} 2 R^{p-1}\|({\phi_2}, {\psi_2}) - ({\phi_1}, {\psi_1})\|_{X \times X}  .  
\end{equation}
By \eqref{Q_1 result} and \eqref{J_1 result}, we can deduce that
\begin{equation}\label{H_2 result}
    \|\hat{\phi}_2-\hat{\phi}_1\|_X \leq  \hat{C}\ \big( \underset{B_\delta}{\sup} \ \kappa_2(x)+ \frac{1}{t\delta^{\xi-1}}+2R^{p-1} \big) \|({\phi_2}, {\psi_2}) - ({\phi_1}, {\psi_1})\|_{X \times X}.
\end{equation}
So by \eqref{H_1 result} and \eqref{H_2 result}, we have
\begin{align*}
&\|T_t(\phi_2,\psi_2)-T_t(\phi_1,\psi_1)\|_{X \times X} = \| \hat{\phi }_1- \hat{\phi}_2\|_X + \| \hat{\psi }_1- \hat{\psi}_2\|_X 
   \\&  \leq C \big( \underset{B_\delta}{\sup}\{\kappa_1+ \kappa_2\}+2R^{p-1} + \frac{1}{t\delta^{\xi-1}} \big)  \|({\phi_2}, {\psi_2}) - ({\phi_1}, {\psi_1})\|_{X \times X}.
\end{align*}
\end{proof}
\subsection{Proof of the main theorem }
\textbf{Proof of Theorem \eqref{main_theorem} } We can now complete the proof of our main theorem. Recall that we want to find some $\phi$ and $\psi$ which satisfy \eqref{linear_syst} such that  $u(x)=w_t(x)+ \phi(x), $ and $  v(x)=w_t(x)+ \psi(x) $ satisfy equation \eqref{eq_syst}. We will show that the mapping $T_t$ is a contraction on $\mathcal{F}_R$ for suitable $0<R<1$, and large $t$.\\ 
\textit{Into}. Let $0<R<1$, $0<\delta<1$, $t>1$, and  $\phi, \psi \in B_R$. Set $T_t(\phi, \psi)= (\hat{\phi}, \hat{\psi})$. Then by Lemma \ref{into_lemma}, there exists some $C>0$ (independent of the parameters) such that 
\begin{align*}
    \| T_t(\psi, \phi)\|_X= \|( \hat{\psi},\hat{ \phi})\|_X \leq &  C \bigg( R^p +\underset{B_\delta}{\sup} \{|\kappa_1|(x)+|\kappa_2|(x)|\} +   \frac{1}{ t^{\frac{p}{p-1} }|\delta|^{(\xi-1)\frac{p}{p-1}}} \bigg) .
\end{align*}
 Hence, for $(\hat{\phi}, \hat{\psi}) $ to be in $\mathcal{F}_R$, it is sufficient to have 
 \begin{equation}\label{into-inequality}C \bigg(  R^p +\underset{B_\delta}{\sup} \{|\kappa_1|(x)+|\kappa_2|(x)|\} +  \frac{1}{ t^{\frac{p}{p-1} }|\delta|^{(\xi-1)\frac{p}{p-1}}} \bigg) \leq R.\end{equation}
\noindent
\textit{Contraction}. Let $0<R<1$, $0<\delta<1$, and   $t>1$, also $\phi_i, \psi_i \in B_R$. Set $T_t(\phi_i, \psi_i)= (\hat{\phi_i}, \hat{\psi_i})$. Then by Lemma \ref{contraction_lemma}  we have 
\begin{align*}
    \|(\hat{\phi_2}, \hat{\psi_2})- (\hat{\phi_1}, \hat{\psi_1})\|_{X \times X}\leq  C \big( \underset{B_\delta}{\sup}(\kappa_1+ \kappa_2)+2R^{p-1} + \frac{1}{t\delta^{\xi-1}} \big)  \|({\phi_2}, {\psi_2}) - ({\phi_1}, {\psi_1})\|_{X \times X}.
\end{align*}
Hence, for $T_t$ to be a contraction, it is sufficient to have \begin{equation}\label{contraction_inequality}
 C \big( \underset{B_\delta}{\sup}(\kappa_1+ \kappa_2) +2R^{p-1}+ \frac{1}{t\delta^{\xi-1}} \big)   \leq \frac{1}{2}.\end{equation} 
 We now choose the parameters. Note that we can satisfy both \eqref{into-inequality} and \eqref{contraction_inequality} by first taking $R>0$ sufficiently small, similarly  we take $\delta>0$ sufficiently small and then we take $t$ large enough. Thus we obtain that $T_t$ is into and a contraction.
We can now apply Banach's Fixed Point Theorem to find the fixed point $(\phi, \psi)$ of $T_t$ on $\mathcal{F}_R$ and hence $(\phi, \psi
)$ satisfies \eqref{linear_syst}. Thus $u_t=u(x)=w_t(x)+ \phi(x)$, and $v_t=v(x)=w_t(x)+ \psi(x)$ are solutions of \eqref{eq_syst} provided $(\phi,\psi)$ satisfies \eqref{linear_syst}. \noindent We need to verify that we can obtain a nonzero positive solution. Note that we have 
\begin{align*}
    |x|^{\sigma+1}|\nabla u_t(x)| \geq r^{\sigma+1} |w_t'(r)|-|x|^{\sigma+1}|\nabla\phi(x)| \geq r^{\sigma+1}|w_t'(r)|-R,\\
    |x|^{\sigma+1}|\nabla v_t(x)| \geq r^{\sigma+1} |w_t'(r)|-|x|^{\sigma+1}|\nabla\psi(x)| \geq r^{\sigma+1}|w_t'(r)|-R
\end{align*}
since $\phi, \psi \in B_R$. Recall that for all positive $t$, we have ($\lim_{r \to 0}r^{\sigma+1}w'_t(r)= -C_\beta$) and hence by taking $R>0$ sufficiently small, we can see $u_t,v_t$ are nonzero when for some small positive $\epsilon$, we have $0<|x|\leq\epsilon$. We want to show that both $u_t$ and $v_t$ are positive on $B_1 \backslash\{0\}$.
%((DOES THIS HOLD FOR ALL $t$ OR IS $t$ FIEXED OR ??? does the C depend on t or ???))
For all positive $t$, in the equation of \eqref{integral_solution}, we can find some $z \in (0,1)$ such that we have $ty^{\xi-1}<1$ for any $y$ that satisfies $0<y<z$.
 Since we have $\frac{-1}{p-1}+1= \sigma$, we get:
\begin{align*}
    w_t=\int_r^1 \frac{1}{(ty^{\xi}+\beta y)^{\frac{1}{p-1}}} dy \geq \int_r^{z} \frac{1}{(y(ty^{\xi-1}+\beta) )^{\frac{1}{p-1}}} dy \geq \int_r^z \frac{1}{ y^{\frac{1}{p-1}}(1+\beta)^{\frac{1}{p-1}}} \geq C  \big[-\sigma  y^{-\sigma}\big]_r^z=C\big[\frac{1}{r^{\sigma}}-\frac{1}{z^{\sigma}}\big] .
\end{align*} 
Thus there exists some positive constant $C=C_{(\beta,p)}$ such that for all $0<r<z$, we have
\[r^{\sigma}w_t \geq {C(1-(\frac{r}{z})}^{\sigma})\] %((THIS HOLDS FOR ALL $r$ OR ????)) 
and we know  $u_t=w_t +\phi $, and $ v_t=w_t +\psi$
where $\phi, \psi $ are in $X$.
Thus near the origin when $r  $ goes to zero, we see that $u_t,v_t$ are positive.
Now for $\epsilon<|x|<1$, note that since $\kappa_i \geq0$, we have $- \Delta u_t = (1+\kappa_1(x)) | \nabla v|^p \geq 0$, and $-\Delta v= (1+\kappa_2(x)) | \nabla u|^p \geq 0$.
\noindent
Thus, by applying the maximum principle, we can say that $u_t$, $v_t$ are both positive on $\epsilon <|x| < 1$. It verifies that we have nonzero positive solutions.
$\hfill \Box$
\\ \\
\noindent 
\textbf{Data Availability} No data were generated or analyzed during the course of this study, so data sharing is not
applicable.
\section*{Deceleration}
\noindent
\textbf{Conflicts of Interest:} The author declares no conflicts of interest.\\

\noindent
\textbf{Ethical approval} This research did not involve any studies with human participants or animals. Therefore,
ethical approval was not required for this study.\\

\noindent
\textbf{Consent} The author agreed with the content of the paper and has given consent to the present submission.

\end{document}